\documentclass[a4paper,10pt]{article}

\usepackage[english]{babel}

\usepackage[]{amsmath}
\usepackage{amssymb}
\usepackage{amsthm}  
\theoremstyle{plain}
\newtheorem{theorem}{Theorem}
\newtheorem{proposition}[theorem]{Proposition}
\newtheorem{lemma}[theorem]{Lemma}
\newtheorem{corollary}[theorem]{Corollary}
\theoremstyle{definition}

\newtheorem{remark}[theorem]{Remark}
\newtheorem{example}[theorem]{Example}

\usepackage{subcaption}

\usepackage{algorithm,algpseudocode}
\usepackage{siunitx}
\usepackage{placeins}
\usepackage{multicol}
\usepackage{hyperref}
\usepackage{graphicx}
\usepackage{tikz}
\usetikzlibrary{arrows,decorations.pathmorphing,decorations.markings,math,arrows.meta,patterns,calc}
\usepackage{cleveref}
\usepackage{makeidx}  
\allowdisplaybreaks

\DeclareMathOperator{\argmax}{argmax}
\DeclareMathOperator{\blkdiag}{blkdiag}

\DeclareMathOperator{\rank}{rank}
\DeclareMathOperator{\diag}{diag}

\newcommand{\minus}{\scalebox{0.75}[1.0]{\( \,-\, \)}}
\newcommand{\plus}{\scalebox{0.75}[1.0]{\( \,+\, \)}}
\newcommand{\negative}{\scalebox{0.75}[1.0]{\( - \)}}
\newcommand{\hinf}{\mathcal{H}_{\infty}}
\newcommand{\hinfnorm}{$\mathcal{H}_{\infty}$-norm}
\def\ie{i.e.}
\def\Ltwo{L^{2}}

\makeatletter
\pgfdeclarepatternformonly[\LineSpace]{my north west lines}{\pgfqpoint{-1pt}{-1pt}}{\pgfqpoint{\LineSpace}{\LineSpace}}{\pgfqpoint{\LineSpace}{\LineSpace}}%
{
	\pgfsetcolor{\tikz@pattern@color}
	\pgfsetlinewidth{0.4pt}
	\pgfpathmoveto{\pgfqpoint{0pt}{\LineSpace}}
	\pgfpathlineto{\pgfqpoint{\LineSpace + 0.1pt}{-0.1pt}}
	\pgfusepath{stroke}
}
\makeatother

\newdimen\LineSpace
\tikzset{
	LineSpace/.code={\LineSpace=#1},
	LineSpace=3pt
}

\newtheorem{assumption}{Assumption}
\theoremstyle{definition}

\crefname{property}{Property}{Properties}
\crefname{corollary}{Corollary}{Corollaries}

	\title{A scalable controller synthesis method for the robust control of networked systems}

\date{}
\author{Pieter Appeltans and Wim Michiels \\ Department of Computer Science, KU Leuven\\ {\tt \{pieter,wim\}.\{appeltans,michiels\}@cs.kuleuven.be}}
\begin{document}

\maketitle
\begin{abstract}
	This manuscript discusses a scalable controller synthesis\linebreak method for networked systems with a large number of identical subsystems based on the H-infinity control framework. The dynamics of the individual subsystems are described by identical linear time-invariant delay differential equations and the effect of transport and communication delay is explicitly taken into account. The presented method is based on the result that, under a particular assumption on the graph describing the interconnections between the subsystems, the H-infinity norm of the overall system is upper bounded by the robust H-infinity norm of a single  subsystem with an additional uncertainty. This work will therefore briefly discuss a recently developed method to compute this last quantity. The resulting controller is then obtained by directly minimizing this upper bound in the controller parameters.
\end{abstract}
{
}
\section{Introduction}
This manuscript presents a controller synthesis method for networked systems. Such networked systems consist of a large number of smaller subsystems that interact over a network. The analysis and control of these networked systems is challenging due to their large dimension and the presence of delays. These delays originate from the time needed to transfer matter, energy and information between subsystems. In this context, the traditional approach of using one global controller for the complete network is thus not feasible due to the high communication requirements and the poor scalability with respect to the number of subsystems. Furthermore, the assumption that all measurements are centrally available, does often not hold for networked systems. These limitations inspired local control approaches, in which each subsystem has its own local controller. Neighboring controllers can however communicate to improve control performance. 

In this manuscript we consider networked systems in which the dynamics of the individual subsystems are described by identical linear time-invariant delay differential equations. The resulting local controllers are identical and minimize an upper bound for the \hinfnorm{} of the overall system. This \hinfnorm{} is an important performance measure in robust control theory, see \cite{pa:Zhou1998}. 

The computation cost of the standard algorithms for calculating the \hinfnorm{} of dynamical systems with discrete delays, such as \cite{pa:Gumussoy2011}, scales cubically with respect to the number of states (and hence the number of subsystems). However, for some networked systems this computation cost can be decreased significantly using the decoupling transformation presented in \cite{pa:Massioni2009} and \cite{pa:Dileep2018a}. More specifically, if the subsystems are identical and the graph describing the interconnections between the subsystems fulfills a particular assumption, then the \hinfnorm{} of the complete system is equal to the maximal \hinfnorm{} of a single parametrized subsystem where the allowable values of the parameter correspond to the eigenvalues of the adjacency matrix of the interconnection graph. Moreover, in \cite{pa:Dileep2018} it was suggested to consider this parameter as an uncertainty bounded to a region in the complex plane that comprises all these eigenvalues. As such, the worst-case \hinfnorm{} of this uncertain subsystem gives an upper bound for the \hinfnorm{} of the complete network. Furthermore, this worst-case \hinfnorm{} is also known as the robust \hinfnorm{} and can be computed at a cost that only depends on the dimension of an individual subsystem using the method presented in \cite{pa:Appeltans2019}.

For the controller synthesis, we will directly minimize the robust \hinfnorm{} of the uncertain subsystem in the controller parameters. Our method thus fits in the frequency based, direct optimization framework, used in \cite{pa:Gumussoy2011},\cite{pa:Michiels2011},\cite{pa:Dileep2018} and \cite{pa:Ozer2015}. This framework allows to easily incorporate constraints on the structure of the controller, such as PID or reduced order control. In contrast, $\hinf$-controller design methods based on Ricatti equations and linear matrix inequalities typically give rise to dense controllers with dimensions equal to that of the system. A notable exception is \cite{pa:Hilhorst2015}, which allows to design reduced order controllers.  Another advantage of the direct optimization approach compared to methods based on Ricatti equations and linear matrix inequalities, in particular for systems with delays, is that the obtained results are less conservative. This comes however at the cost of having to solve a non-convex and non-smooth optimization problem.

The remainder of this work is structured as follows. First, Section~\ref{sec:networked_systems} introduces the considered networked systems and details the aforementioned decoupling transformation. Next, a recently developed method to compute the robust \hinfnorm{} of uncertain linear time-invariant systems with discrete delays is discussed in Section~\ref{pa_sec:computing_rob_hinf}. Subsequently, the direct optimization approach to synthesize the controller is outlined in Section~\ref{pa_sec:controller_design}. Finally, the resulting design methodology is illustrated using an example problem in Section~\ref{pa_sec:example} and some concluding remarks are given in Section~\ref{pa_sec:conclusion}.

\section{Computing the \hinfnorm{} of networked systems}
\label{sec:networked_systems}
In Section~\ref{pa_subsec:system_description_control_objective} we introduce the considered networked systems and the control objective. Section~\ref{pa_subsec:decoupling_transformation} presents the decoupling transformation that allows to compute an upper bound for the \hinfnorm{} of the overall system at a computation cost that does not depend on the number of subsystems. 

\subsection{System Description and Control Objective}
\label{pa_subsec:system_description_control_objective}
 Here, we consider networked systems with $N$ subsystems. The dynamics of the individual subsystems are identical and described by a state-space representation of the following form:
\begin{equation}
\label{pa_eq:subsysem}
\left\{
\setlength{\arraycolsep}{2pt}
\def\arraystretch{1.3}
\begin{array}{rcl} 
	\dot{x}_{j}(t) &=& \sum\limits_{k=0}^{K} A_{k}\, x_{j}(t-\tau_k) + B_{u}\, u_{j}(t-\tau_{u}) + B_{u_{n}}\,u^{n}_{j}(t)+ B_w\, w_{j}(t) \\
	y_{j}(t)&=& C_{y_{\phantom{n}}}\, x_{j}(t)\\
	y^{n}_{j}(t)&=& C_{y_{n}}\, x_{j}(t)\\
	z_{j}(t) &=& C_{z_{\phantom{n}}}\, x_{j}(t) \hspace{13.4em} \text{for $j = 1,\dotsc,N$}
\end{array}
\right.
\end{equation}
with $x_{j}(t)\in\mathbb{R}^{n}$ the state vector of subsystem $j$, $u_{j}(t)\in\mathbb{R}^{m_c}$ its control input, $y_{j}(t) \in \mathbb{R}^{p_c}$ its measured output, $w_{j}(t)\in\mathbb{R}^{m}$ its performance input, $z_{j}(t)\in\mathbb{R}^{p}$ its performance output, $ 0 = \tau_0 < \tau_1 < \dots < \tau_K$ discrete delays, $\tau_u\geq0$ an input delay and $A_k$, $B_{u}$, $B_{u_{n}}$, $B_w$, $C_{y}$, $C_{y_{n}}$ and $C_{z}$ \mbox{real-valued} matrices of appropriate dimension. The input $u^{n}_{j}(t)$ and the output $y^{n}_{j}(t)$ model the interactions between the subsystems:
\[
	u^{n}_{j}(t) = \sum_{i=1}^{N} P^{N}_{j,i}~y^{n}_{i}(t-\tau_{n})\text{,}
\] with $\tau_n\geq0$ the interaction delay and $P_{N} = [P^{N}_{j,i}]_{j,i=1}^{N}$ the adjacency matrix of the interconnection graph of the network. More specifically, an element $P^{N}_{j,i}$ is non-zero if and only if the dynamics of subsystem $j$ are influenced by \mbox{subsystem $i$}.  

Each subsystem is controlled using a local controller and these local controllers are identical. Here we will consider dynamic output feedback controllers of order $n_c$: 
\begin{equation}
\label{pa_eq:controller}
\left\{
\def\arraystretch{1.3}
\begin{array}{rcl}
	\dot{\xi}_{j}(t) &=& J_\mathbf{p}~\xi_{j}(t) + F_{\mathbf{p}}~y_{j}(t) +  F^{n}_{\mathbf{p}}~u^{nc}_{j}(t)\\
	u_{j}(t) &=& L_{\mathbf{p}}~\xi_{j}(t) + K_{\mathbf{p}}~y_{j}(t) +  K^{n}_{\mathbf{p}}~ u^{nc}_{j}(t) \qquad \text{for $j = 1,\dots,N$}
\end{array}	
\right.	
\end{equation}
with $\xi_{j}(t) \in \mathbb{R}^{n_{c}}$ the controller state of the local controller associated with subsystem $j$. The matrices $J_\mathbf{p}$, $F_{\mathbf{p}}$, $F^{n}_{\mathbf{p}}$, $ L_{\mathbf{p}}$, $K_{\mathbf{p}}$ and $K^{n}_{\mathbf{p}}$ are real-valued and of appropriated dimension. The subscript $\mathbf{p}$ is used to indicate that these matrices depend on some tunable control parameters $\mathbf{p}$. If $F^{n}_{\mathbf{p}} \neq 0$ and/or $K^{n}_{\mathbf{p}} \neq 0$, the local controllers can communicate their sensor measurements to neighboring subsystems. It is however required that the adjacency matrix of the communication graph is the equal to the adjacency matrix of the interaction graph:
\[
u^{nc}_{j}(t) = \textstyle\sum\limits_{i=1}^{N} P^N_{j,i}~y_{i}(t-\tau_{nc}),
\] 
with $\tau_{nc}\geq0$ the communication delay.

\begin{remark}
	In the remainder of this manuscript we restrict our attention to controller architectures where the local controllers can only share sensor measurements. Note however that the results can be extended to architectures where the local controllers can also share their internal state. 
\end{remark}

By eliminating the control and coupling variables, we find the following state-space description for the closed-loop of the complete networked system:
\begin{equation}
\label{pa:state_space_overall}
\left\{
\def\arraystretch{1.3}
\setlength{\arraycolsep}{2pt}
\begin{array}{rcl}
 \dot{x}(t)& =&  \sum\limits_{k=0}^{K} (I_N \otimes A_k)\,x(t\minus\tau_k) + \left(I_N \otimes B_{u} K_{\mathbf{p}} C_{y}\right)\,x(t\minus\tau_{u}) \, + \\
 && \left(P_{N}\otimes B_{u_{n}}C_{y_{n}}\right)\,x(t\minus\tau_{n}) + \big(I_{N}\otimes B_{u}L_{\mathbf{p}}\big)\,\xi(t\minus\tau_{u}) \,+  \\
 && \left(P_{N}\otimes B_{u}K^{n}_{\mathbf{p}}C_{y}\right)\,x(t\minus\tau_{u}\minus\tau_{nc})  + \left(I_N\otimes B_w\right)\,w(t) \\ 
 \dot{\xi}(t)& =& \left(I_N \otimes J_{\mathbf{p}}\right)\,\xi(t) + (I_N \otimes F_{\mathbf{p}}C_{y})\,x(t) \,+ \\ &&\left(P_N \otimes F^{n}_{\mathbf{p}} C_{y}\right)\,x(t\minus\tau_{nc})\\ 
 z(t)& =& (I_{N} \otimes C_{z})\,x(t)
\end{array}
\right. 
\end{equation}
with $I_N$ the identity matrix of size $N$, $x(t) = [x_1(t)^{T} \dotsm~ x_N(t)^{T}]^{T}$ the combined state, $\xi(t) = [\xi_1(t)^{T} \dotsm~ \xi_N(t)^{T}]^{T}$ the combined controller state, $w(t) = [w_{1}(t)^{T} \dotsm~ w_{N}(t)^{T}]^{T}$ the combined performance input,\linebreak $z(t) = [z_{1}(t)^{T} \dotsm~ z_{N}(t)^{T}]^{T}$ the combined performance output and $\otimes$ the Kronecker product. The corresponding transfer function from $w$ to $z$ is equal to:
\begin{equation}
\label{pa_eq:overall_tf}
\setlength{\arraycolsep}{2pt}
\begin{array}{rcc}
T(s;\mathbf{p},N) & = & \left(I_{N}\otimes \begin{bmatrix}  C_{z}&  0 \end{bmatrix}\right) \left(I_{N(n+n_c)}s \minus  I_{N} \otimes Q_{\mathbf{p}}(s) \minus P_{N} \otimes R_{\mathbf{p}}(s)\right)^{\negative 1}\times \\[7pt]
&& \Big(I_{N}\otimes\begin{bmatrix}
B_w \\
0
\end{bmatrix}\Big)
\end{array}	
\end{equation}
with 
\begin{align*}
Q_{\mathbf{p}}(s) = & \begin{bmatrix}
 A_0 & 0 \\
F_{\mathbf{p}}C_{y} & J_{\mathbf{p}}
\end{bmatrix} + \sum_{k=1}^{K}\begin{bmatrix} A_k & 0 \\
0 & 0
\end{bmatrix}e^{-s \tau_k} +  \begin{bmatrix}
 B_{u} K_{\mathbf{p}} C_{y} & B_{u}L_{\mathbf{p}}\\
 0 & 0
\end{bmatrix} e^{\negative s\tau_{u}} 
\intertext{and} 
R_{\mathbf{p}}(s) = & \setlength{\arraycolsep}{2pt}\begin{bmatrix}
B_{u_{n}}C_{y_{n}} & 0 \\
0 & 0
\end{bmatrix}e^{\negative s\tau_{n}} + \begin{bmatrix}
B_{u}K^{n}_{\mathbf{p}}C_{y} & 0\\
0 & 0
\end{bmatrix}e^{\negative s(\tau_{u}+\tau_{nc})} \, +   \begin{bmatrix}
0 & 0\\
F^{n}_{\mathbf{p}} C_{y} & 0
\end{bmatrix}e^{\negative s\tau_{nc}}.
\end{align*}

If system \eqref{pa:state_space_overall} is exponentially stable, the \hinfnorm{} of \eqref{pa_eq:overall_tf} equals:
\begin{equation*}
\|T(\cdot;\mathbf{p},N)\|_{\mathcal{H}_{\infty}} = \max_{\omega \in \mathbb{R}^{+}} \sigma_1\big(T(\jmath\omega;\mathbf{p},N)\big)
\end{equation*}
with $\sigma_1(\cdot)$ the largest singular value of its matrix argument \cite{pa:Gumussoy2011}. 
Here we recall that the \hinfnorm{} is an important performance measure in robust control theory, used to asses the disturbance rejection of a dynamical system as it gives the worst-case energy gain of the system with respect to energy-bounded noise signals:
\[
\|T(\cdot;\mathbf{p},N)\|_{\mathcal{H}_{\infty}} = \max_{w \in \Ltwo_m} \frac{\|z\|_{\Ltwo_p}}{\|w\|_{\Ltwo_m}},
\]
with $\|w\|_{\Ltwo_m} = \sqrt{\int_{0}^{+\infty} \|w(t)\|_2^{2}\> dt}$, $\Ltwo_m=\{w : [0,+\infty) \mapsto \mathbb{R}^{m} \text{ such that } \linebreak \|w\|_{\Ltwo_m}^{2} < + \infty \}$,   $\|w(t)\|_2$ the Euclidean norm, and $\|z\|_{\Ltwo_p}$ and $\Ltwo_p$ defined analogously \cite{pa:Zhou1998}.

%

\subsection{The robust \hinfnorm{} of subsystem as upper bound for the \hinfnorm{} of the overall network}
\label{pa_subsec:decoupling_transformation}
In this subsection we show that under the following assumption on $P_N$, there exists a decoupling transformation that allows to compute an upper bound for the \hinfnorm{} of \eqref{pa_eq:overall_tf} at a computation cost that does not depend on the number of subsystems.   
\begin{assumption}
	\label{pa_ass:adjacency_matrix_unitary}
	The matrix $P_{N}$ has real-valued eigenvalues confined to an interval $[a,\,b]$ and is diagonalizable by a unitary matrix $V_{N}$, \ie{}
	\[
	{V_{N}}^{H}P_{N}V_{N} = \Lambda_{N} \text{,}
	\]
	with $\Lambda_{N} = \diag(\lambda_1,\lambda_2,\dots,\lambda_N)$ and $\lambda_j\in [a,\,b]$ for $j=1,\dots,N$.
\end{assumption}
If we apply the following change of variables to the states, the controller states, the performance input and the performance output of system \eqref{pa:state_space_overall}
\begin{align*}
	\bar{x}(t) &= ({V_N}^{H} \otimes \rlap{$I_{n}$}\hphantom{I_{n_c}} )~\hphantom{w}\llap{$x$}(t)\\
	\bar{\xi}(t) &= ({V_N}^{H} \otimes I_{n_c})~\hphantom{w}\llap{$\xi$}(t) \\
	\bar{w}(t) &= ({V_N}^{H} \otimes \rlap{$I_{m}$}\hphantom{I_{n_c}} )~w(t) \\
	\bar{z}(t) &= ({V_N}^{H} \otimes \rlap{$I_{p}$}\hphantom{I_{n_c}} )~\hphantom{w}\llap{$z$}(t)
\end{align*}
 we obtain 
\begin{equation}
\label{pa_eq:transformed_system}
\left\{
\setlength{\arraycolsep}{2pt}
\def\arraystretch{1.3}
\begin{array}{rcl}
\dot{\bar{x}}(t)& =&  \sum\limits_{k=0}^{K} (I_N \otimes A_k)\,\bar{x}(t\minus\tau_k)+ (I_N \otimes B_{u}K_{\mathbf{p}}C_{y})\,\bar{x}(t\minus\tau_{u}) + \phantom{x}
\\ && (\Lambda_{N} \otimes B_{u_{n}}C_{y_{n}})\,\bar{x}(t\minus\tau_{n}) + \left(I_{N}\otimes B_{u}L_\mathbf{p}\right)\,\bar{\xi}(t\minus\tau_{u})  + \\
&& (\Lambda_{N}\otimes B_{u}K^{n}_{\mathbf{p}}C_{y})\,\bar{x}(t\minus\tau_{u}\minus\tau_{nc})  + (I_N \otimes B_w)~\bar{w}(t) \\
\dot{\bar{\xi}}(t)& =& \left(I_N \otimes J_\mathbf{p}\right)\, \bar{\xi}(t) + \left(I_N \otimes F_\mathbf{p}C_{y}\right)\,\bar{x}(t)+ \\ &&  \left(\Lambda_{N} \otimes F^{n}_\mathbf{p} C_{y}\right) ~\bar{x}(t-\tau_{nc})\\ 
\bar{z}(t)& =& (I_{N} \otimes C_{z})~\bar{x}(t) \text{.}
\end{array}
\right.
\end{equation}
Notice that all matrices in \eqref{pa_eq:transformed_system} are block diagonal and hence the behavior of this transformed system is fully characterized by its $N$ independent subsystems. This leads to the following theorem.
\begin{theorem}
	\label{th:decoupling_H_infinity}
	For a networked system of form \eqref{pa:state_space_overall} whose adjacency matrix fulfills \Cref{pa_ass:adjacency_matrix_unitary}, it holds that
	
	\begin{equation*}
	\|T(\cdot;\mathbf{p},N)\|_{\mathcal{H}_{\infty}} = \|\bar{T}_{\bar{w}\bar{z}}(\cdot;\mathbf{p},N)\|_{\mathcal{H}_{\infty}} = \max_{\lambda \in \{\lambda_1,\dots,\lambda_N\}} \|\hat{T}_{\hat{w}\hat{z}}(\cdot;\mathbf{p},\lambda) \|_{\mathcal{H}_{\infty}}
	\end{equation*}
	with $\bar{T}_{\bar{w}\bar{z}}(\cdot;\mathbf{p},N)$ the transfer function from $\bar{w}$ to $\bar{z}$ of system \eqref{pa_eq:transformed_system} and \linebreak $\hat{T}_{\hat{w}\hat{z}}(\cdot;\mathbf{p},\lambda)$ the transfer function from $\hat{w}$ to $\hat{z}$ of the following system parameterized in $\lambda$:
	
	\begin{equation}
	\label{pa:decoupled_system}
	\left\{
	\arraycolsep=2pt
	\def\arraystretch{1.4}
	\begin{array}{rcl}
		\dot{\hat{x}}(t)& =& \sum\limits_{k=0}^{K} 
		A_k\,
		\hat{x}(t\minus\tau_k) 
		 + 
		B_{u}K_{\mathbf{p}}C_{y}\,\hat{x}(t\minus\tau_{u}) + \lambda 
		B_{u_{n}}C_{y_{n}}\,\hat{x}(t\minus\tau_{n}) \\
		&& 
		+  \lambda B_{u}K^{n}_{\mathbf{p}}C_{y}\,	\hat{x}(t\minus\tau_{u}\minus\tau_{nc})+B_{u}L_{\mathbf{p}}\,\hat{\xi}(t\minus\tau_{u})+ B_w\,\hat{w}(t)\\
		 \dot{\hat{\xi}}(t)& =& J_{\mathbf{p}}\,\hat{\xi}(t) + F_{\mathbf{p}}C_{y}\,\hat{x}(t) + \lambda F^{n}_{\mathbf{p}}C_{y}\, \hat{x}(t\minus\tau_{nc}) \\
		\hat{z}(t)& =&  C_{z}\,\hat{x}(t) \text{.}
	\end{array}
	\right.
	\end{equation}
\end{theorem}
\begin{proof}
	The provided proof is added for self-containedness and is similar to the ones given in \cite{pa:Massioni2009} and \cite{pa:Dileep2018a}. We refer to these papers for more details. \\ 
	The relation between $T(\jmath\omega;\mathbf{p},N)$ and $\bar{T}_{\bar{w}\bar{z}}(\jmath\omega;\mathbf{p},N)$ is given by
	\[
	T(\jmath\omega;\mathbf{p},N) = \left({V_{N}}\otimes I_p\right)~\bar{T}_{\bar{w}\bar{z}}(\jmath\omega;N)~\left({V_{N}}^{H} \otimes I_m\right)\text{.}
	\]
	Because $({V_{N}}\otimes I_p)$ and $({V_{N}}^{H}\otimes I_m)$ are unitary matrices if follows that,  
	\[
	\begin{aligned}
	\sigma_1\big(T(\jmath\omega;\mathbf{p},N)\big) &= \sigma_1\Big(\left({V_{N}}\otimes I_p\right)~\bar{T}_{\bar{w}\bar{z}}(\jmath\omega;\mathbf{p},N)~\left({V_{N}}^{H} \otimes I_m\right)\Big) \\ &= \sigma_1\left(\bar{T}_{\bar{w}\bar{z}}(\jmath\omega;\mathbf{p},N)\right).
	\end{aligned}
	\] 
	The second equality follows from the fact that 
	\[
	\bar{T}_{\bar{w}\bar{z}}(\jmath\omega;\mathbf{p},N) = \blkdiag_{j=1,\dots,N}\big(\hat{T}_{\hat{w}\hat{z}}(\jmath\omega;\mathbf{p},\lambda_j)\big) \text{,}
	\]
	and hence
	\[
	\sigma_1\left(\bar{T}_{\bar{w}\bar{z}}(\jmath\omega;\mathbf{p},N)\right) = \max_{\lambda\in\{\lambda_1,\dots,\lambda_N\}} \sigma_1\big(\hat{T}_{\hat{w}\hat{z}}(\jmath\omega;\mathbf{p},\lambda)\big),
	\]
	which concludes the proof.
\end{proof}
	Note that system \eqref{pa:decoupled_system} corresponds to a single subsystem in \eqref{pa:state_space_overall} where the network connections are replaced by a parameter.
By treating $\lambda$ in \eqref{pa:decoupled_system} as an uncertainty confined to the interval $[a,\,b]$, the robust \hinfnorm{} associated with \eqref{pa:decoupled_system}, which is defined as the maximal value of the \hinfnorm{} over all instances of the uncertain parameter,
\begin{equation}
\label{pa_eq:rob_hinf_decoupled}
\|\hat{T}_{\hat{w}\hat{z}}(\cdot;\mathbf{p},\cdot)\|_{\mathcal{H}_{\infty}}^{[a,\,b]} = \max_{\lambda \in [a,\,b]} \|\hat{T}_{\hat{w}\hat{z}}(\cdot;\mathbf{p},\lambda)\|_{\mathcal{H}_{\infty}}\text{,}
\end{equation}
can be used as an upper bound for the \hinfnorm{} of \eqref{pa:state_space_overall}, as stated in the following corollary.
\begin{corollary}
	\label{pa_cor:upperbound_hinfnorm}
	The $\mathcal{H}_{\infty}$-norm of a networked system of form \eqref{pa:state_space_overall} whose adjacency matrix fulfills \Cref{pa_ass:adjacency_matrix_unitary}, is upper bounded by the robust $\mathcal{H}_{\infty}$-norm of $\hat{T}_{\hat{w}\hat{z}}(\cdot;\mathbf{p},\lambda)$, with $\lambda$ an uncertain parameter confined to $[a,\,b]$:
	\[
	\|T(\cdot;\mathbf{p},N)\|_{\mathcal{H}_{\infty}}  \leq  \|\hat{T}_{\hat{w}\hat{z}}(\cdot;\mathbf{p},\cdot)\|_{\mathcal{H}_{\infty}}^{[a,\,b]}.
	\]
	Furthermore, if \Cref{pa_ass:adjacency_matrix_unitary} holds with $a$ and $b$ independent of $N$, then \linebreak $\|\hat{T}_{\hat{w}\hat{z}}(\cdot;\mathbf{p},\cdot)\|_{\mathcal{H}_{\infty}}^{[a,\,b]}$ is also an upper bound for the supremum of  $\|T(\cdot;\mathbf{p},N)\|_{\mathcal{H}_{\infty}}$ over the number of subsystems:
	\[
	\sup_{N=1,\dots,+\infty} \|T(\cdot;\mathbf{p},N)\|_{\mathcal{H}_{\infty}} \leq  \|\hat{T}_{\hat{w}\hat{z}}(\cdot;\mathbf{p},\lambda) \|_{\mathcal{H}_{\infty}}^{[a,\,b]}.
	\]
	If, furthermore, the (Hausdorff) distance between $[a,\,b]$ and \linebreak $\bigcup\limits_{N=1}^{+\infty} \left\{\lambda\in \mathbb{C} :  \det(I_{N}\lambda-P_{N}) = 0\right\}$ goes to zero then
	\[
	\sup_{N=1,\dots,+\infty} \|T(\cdot;\mathbf{p},N)\|_{\mathcal{H}_{\infty}} =  \|\hat{T}_{\hat{w}\hat{z}}(\cdot;\mathbf{p},\lambda) \|_{\mathcal{H}_{\infty}}^{[a,\,b]} \text{.}
	\]
\end{corollary}

\begin{example}
	\label{pa_ex:adjacency}
	To illustrate the applicability of this result, we consider the following adjacency matrices:
	\[
		P_N^{\,\text{ring}} = 	\setlength{\arraycolsep}{2pt} \begin{bmatrix}
		0 & 0.5 &  &  & &  0.5 \\
		0.5 & 0 & 0.5 &  &  &  \\
		 & 0.5 & 0 & 0.5 &  &  \\
		 &  & \ddots & \ddots & \ddots &  \\
		 &  &  & 0.5 & 0 & 0.5 \\
		0.5 &  &  &  & 0.5 & 0
		\end{bmatrix} \text{ and }
		P_N^{\,\text{line}} = \begin{bmatrix}
		0 & 0.5 &  &  &  &   \\
		0.5 & 0 & 0.5 &  &  &  \\
		 & 0.5 & 0 & 0.5 &   & \\
		 &  & \ddots & \ddots & \ddots &  \\
		 &  &  & 0.5 & 0 & 0.5 \\
		 &  &  &  & 0.5 & 0
		\end{bmatrix}.\hfil
	\]
	 The first adjacency matrix, $P_N^{\text{ring}}$, corresponds to a bidirectional ring topology,  see \Cref{pa_fig:ring_network}; the second one corresponds to a bidirectional line topology, see \Cref{pa_fig:line_network}. The eigenvalues of these adjacency matrices are \linebreak $\left\{\cos\left(\frac{2\pi (j-1)}{N}\right)\right\}_{j=1}^{\lfloor\frac{N+2}{2}\rfloor}$ and $\left\{\cos\left(\frac{j\pi}{N+1}\right)\right\}_{j=1}^{N}$, respectively. The eigenvalues of both matrices are thus confined to the interval $[a,\ b] = [-1,\ 1]$ for all $N > 1$. We can therefore apply \Cref{pa_cor:upperbound_hinfnorm} to compute an upper bound for the \hinfnorm{} of \eqref{pa_eq:overall_tf} that holds for all $N > 1$ at a computation cost that only depends on the dimension of a single subsystem. Furthermore, this upper bound is the same for both topologies.
	 
	\begin{figure}
		\begin{minipage}{0.49\linewidth}
			\centering
			\begin{tikzpicture}
			\node[circle,draw,minimum size=35pt] (1) at (60:2){1};
			\node[circle,draw,minimum size=35pt] (2) at (0:2){2};
			\node[circle,draw,minimum size=35pt] (3) at (-60:2){3};
			\node[circle,draw,minimum size=35pt] (4) at (-120:2){\dots};
			\node[circle,draw,minimum size=35pt] (5) at (180:2){$N\minus1$};
			\node[circle,draw,minimum size=35pt] (6) at (120:2){$N$};
			\draw[<->,bend left=90] (1) -- (2);
			\draw[<->] (2) -- (3);
			\draw[<->] (3) -- (4);
			\draw[<->] (4) -- (5);
			\draw[<->] (5) -- (6);
			\draw[<->] (6) -- (1);
			\end{tikzpicture}
		\end{minipage}\hfill
		\begin{minipage}{0.49\linewidth}
			\centering
			\begin{tikzpicture}
			\node[circle,draw,minimum size=35pt] (1) at (0,0){1};
			\node[circle,draw,minimum size=35pt] (2) at (2,0){\dots};
			\node[circle,draw,minimum size=35pt] (3) at (4,0){$N$};
			\draw[<->] (1) -- (2);
			\draw[<->] (2) -- (3);
			\end{tikzpicture}
		\end{minipage}
	\begin{minipage}{.45\textwidth}
		\caption{Bidirectional ring topology}
		\label{pa_fig:ring_network}
	\end{minipage}\hfill
	\begin{minipage}{.45\textwidth}
	\caption{Bidirectional line topology}\label{pa_fig:line_network}
	\end{minipage}
	\end{figure}
\end{example}

\section{Computing the robust \hinfnorm{}}
\label{pa_sec:computing_rob_hinf}
This section introduces a numerical algorithm to efficiently compute the robust \hinfnorm{} of an uncertain linear time-invariant system with discrete delays:
\begin{equation}
\label{pa_eq:uncertain_system}
\left\{
\begin{array}{rcl}
\dot{x}(t) &=& \sum\limits_{r=0}^{R} \left(H_r+ \lambda~G_r \right)~x(t-\tau_r) + B_{w}~w(t) \\ [9pt]
z(t) &=& C_{z} x(t)
\end{array}
\right.
\end{equation}
with $x(t)\in\mathbb{R}^{n}$ the state, $w(t)\in\mathbb{R}^{m}$ performance input, $z(t)\in\mathbb{R}^{p}$ the performance output, $0 = \tau_0 < \tau_1 < \dots < \tau_R$ discrete delays, $H_r$, $G_r$, $B_{w}$ and $C_{z}$ real-valued matrices of appropriate dimension and $\lambda$ a real-valued, scalar parameter addressed as an uncertainty confined to the interval $[a,\,b]$. Note that system \eqref{pa:decoupled_system} fits this form.

Under the assumption that system \eqref{pa_eq:uncertain_system} is internally exponentially stable for all $\lambda \in [a,\,b]$, its (asymptotic) input-output behavior for each allowable value of $\lambda$ is described in the Laplace domain  by the following transfer function:
\[T(s;\lambda) = C_{z}\Big(I_n s-\sum\limits_{r=0}^{R} \left(H_r +\lambda G_r\right)~e^{-s\tau_r} \Big)^{-1}B_w\text{,}\] 
and the associated robust $\mathcal{H}_{\infty}$-norm is equal to
\begin{equation}
\label{pa_eq:rob_hinfnorm}
\|T(\cdot;\cdot)\|_{\mathcal{H}_{\infty}}^{[a,\,b]} = \max_{\lambda \in [a,\,b]} \|T(\cdot;\lambda)\|_{\mathcal{H}_{\infty}} = \max_{\substack{\lambda \in [a,\,b]\\ \omega \in \mathbb{R}^{+}}} \sigma_1\big(T(\jmath\omega;\lambda)\big)\text{.}
\end{equation}

In \cite{pa:Appeltans2019} a novel numerical algorithm to compute the robust $\mathcal{H}_{\infty}$-norm of an uncertain time-delay system is presented. This algorithm is based on the relation between the robust $\hinf$-norm and the robust stability radius of an ``uncertain'' characteristic matrix. This relation is illustrated in \Cref{pa_subsec:dist_ins}. The resulting algorithm is given in \Cref{pa_subsec:numerical_algorithm}.

\subsection{Relation with the robust stability radius}
\label{pa_subsec:dist_ins}
 Consider the following ``uncertain'' characteristic matrix
 \begin{equation}
 \label{pa_eq:uncertain_DEP}
 M(s;\lambda,\Delta) := I_n s  - \textstyle\sum\limits_{r=0}^{R} \left(H_r + \lambda G_r \right)~e^{-s \tau_r} - B_w\Delta C_z
 \text{,}
 \end{equation}
with $H_r$, $G_r$, $B_w$, $C_z$, $\tau_r$ and $\lambda$ as defined above, $I_n$  the identity matrix of size $n$ and $\Delta \in \mathbb{C}^{m \times p}$ a complex-valued uncertainty with $\|\Delta\|_2 \leq \varepsilon$ and $\varepsilon\geq 0$. Note that this uncertain characteristic matrix has two uncertainties: a scalar $\lambda$ which is real-valued and bounded to the interval $[a,\,b]$ and a $m\times p$ matrix $\Delta$ which is complex-valued and bounded in spectral norm by $\varepsilon$, or in other words $\Delta\in \mathcal{B}_{\|\cdot\|_2\leq\varepsilon}^{\mathbb{C}^{m\times p}}$ with \[\mathcal{B}_{\|\cdot\|_2\leq\varepsilon}^{\mathbb{C}^{m\times p}} := \left\{\Delta \in \mathbb{C}^{m\times p} : \|\Delta\|_2 \leq \varepsilon \right\}.\] 

Next, we define three important concepts related to this uncertain characteristic matrix. The spectral value set of this uncertain characteristic matrix is defined as
\[
\Lambda_{\varepsilon}^{[a,\,b]} := \bigcup_{\lambda\in[a,\,b]} ~\bigcup_ {\Delta\in  \mathcal{B}_{\|\cdot\|_2\leq\varepsilon}^{\mathbb{C}^{m\times p}}}\left\{s \in \mathbb{C} :\det\left(M(s;\lambda,\Delta)\right) = 0 \right\}.
\]
 Note that this spectral value set is symmetric with respect to the imaginary axis. The pseudo-spectral abscissa is defined as the real part of the right-most point, i.e. the point with the largest real part, in this spectral value set,
\[
\alpha_{\varepsilon}^{[a,\,b]} := \max\left\{\Re\left(s\right): s \in \Lambda_{\varepsilon}^{[a,\,b]} \right\} \text{.}
\]
Finally, the robust stability radius is defined as the smallest $\varepsilon$ for which this this pseudo-spectral abscissa becomes non-negative,
\begin{equation}
\label{pa_eq:dist_ins}
r_{[a,\,b]} := \min\left\{\varepsilon\in \mathbb{R}_{\geq 0}: \alpha_{\varepsilon}^{[a,\,b]} \geq  0 \right\} \text{.}
\end{equation}
Furthermore, because $\alpha_{\varepsilon}^{[a,\,b]}$ is a continuous function of $\varepsilon$ (this can be shown using a similar argument as in \cite[Section~IV]{pa:Borgioli2019}), the transition to a non-negative pseudo-spectral abscissa is characterized by an $\varepsilon$ for which $\alpha_{\varepsilon}^{[a,\,b]}$ equals zero. This means that the robust stability radius can also be defined as the smallest $\varepsilon$ for which the spectral value set touches the imaginary axis:
\begin{equation}
\label{pa_eq:dist_ins2}
r_{[a,\,b]} = \min\left\{\varepsilon\in \mathbb{R}_{\geq 0}: \exists\, \omega \in \mathbb{R}^{+} \text{~such that~} \jmath\omega \in \Lambda_{\varepsilon}^{[a,\,b]} \right\}.
\end{equation}
The following example illustrates these three concepts in more detail.
\begin{example}
	Consider the following uncertain characteristic matrix:
	\footnotesize
	\begin{equation}
	\label{pa_eq:example_characteristic_matrix}
	\begin{array}{c}
	\setlength{\arraycolsep}{2pt}
	 \begin{bmatrix}
	 1 & 0 \\
	 0 & 1
	 \end{bmatrix} s\minus \! \left(\begin{bmatrix}
	 \negative 5 & 3\\ 2 & \negative 6
	 \end{bmatrix}\!+
	 \lambda 
	 \begin{bmatrix}
	 2 & 2 \\ 
	 \negative 2 & \negative 1
	 \end{bmatrix}
	 \right)\!
	 \minus\!
	 \left(
	 \begin{bmatrix}
	 \negative 3 & \negative 1 \\ 
	 0 & 2
	 \end{bmatrix}\!
	 + \lambda
	 \begin{bmatrix}
	 1  & 1 \\ 
	 \negative 1 & 1
	 \end{bmatrix}
	 \right) e^{\negative s} 
	 \minus \begin{bmatrix}
	 1 \\
	 \negative 3
	 \end{bmatrix}
	 \Delta
	 \begin{bmatrix}
	 2 & 5
	 \end{bmatrix}.
	 \end{array}
	\end{equation}
	\normalsize
	\Cref{pa_fig:spectral_value_set} shows the part of the spectral value set in the region $[-0.5,-0.2]\times\jmath[2.1,2.7]$ for $[a,\,b]$ equal to $[-1,\,1]$ and several values of $\varepsilon$. For $\varepsilon = 0$, only the real-valued uncertainty $\lambda$ plays a role and the spectral value set is a curve. For nonzero $\varepsilon$, also the complex-valued uncertainty $\Delta$ affects the characteristic matrix and the spectral value set becomes a region in the complex plane which grows as $\varepsilon$ increases. \Cref{pa_fig:pseudo_spectral_abscissa} shows the pseudo-spectral abscissa $\alpha_{\varepsilon}^{[a,\,b]}$ in function of $\varepsilon$. We find that $r_{[-1,1]} = 0.22491$. Finally, \Cref{pa_fig:spectral_value_set2} shows the part of the associated spectral value sets in the region $[-3,0.3]\times \jmath [-10,10]$. One sees that the spectral value set touches the imaginary axis at the origin ($\omega = 0$).
		\begin{figure}[!htbp]
		\centering
		\includegraphics[width=0.5\linewidth]{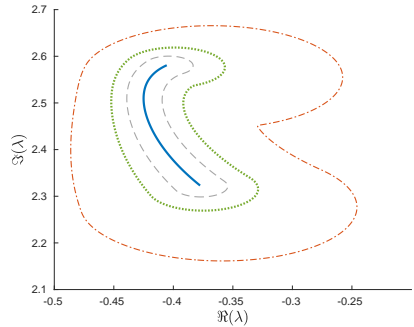}
		\caption{The part of the spectral value set of \eqref{pa_eq:example_characteristic_matrix} in the region $[-0.5,-0.2]\times\jmath[2.1,2.7]$ for $[a,\,b]$ equal to $[-1,\,1]$ and $\varepsilon$ equal to $0$ (full line), $0.05$ (dashed line), $0.1$ (dotted line) and $2/9$ (dash dotted line).}
		\label{pa_fig:spectral_value_set}
	\end{figure} 
	\begin{figure}
		\centering
		\begin{minipage}[b]{0.49\linewidth}
			\centering
			\includegraphics[width=0.9\linewidth]{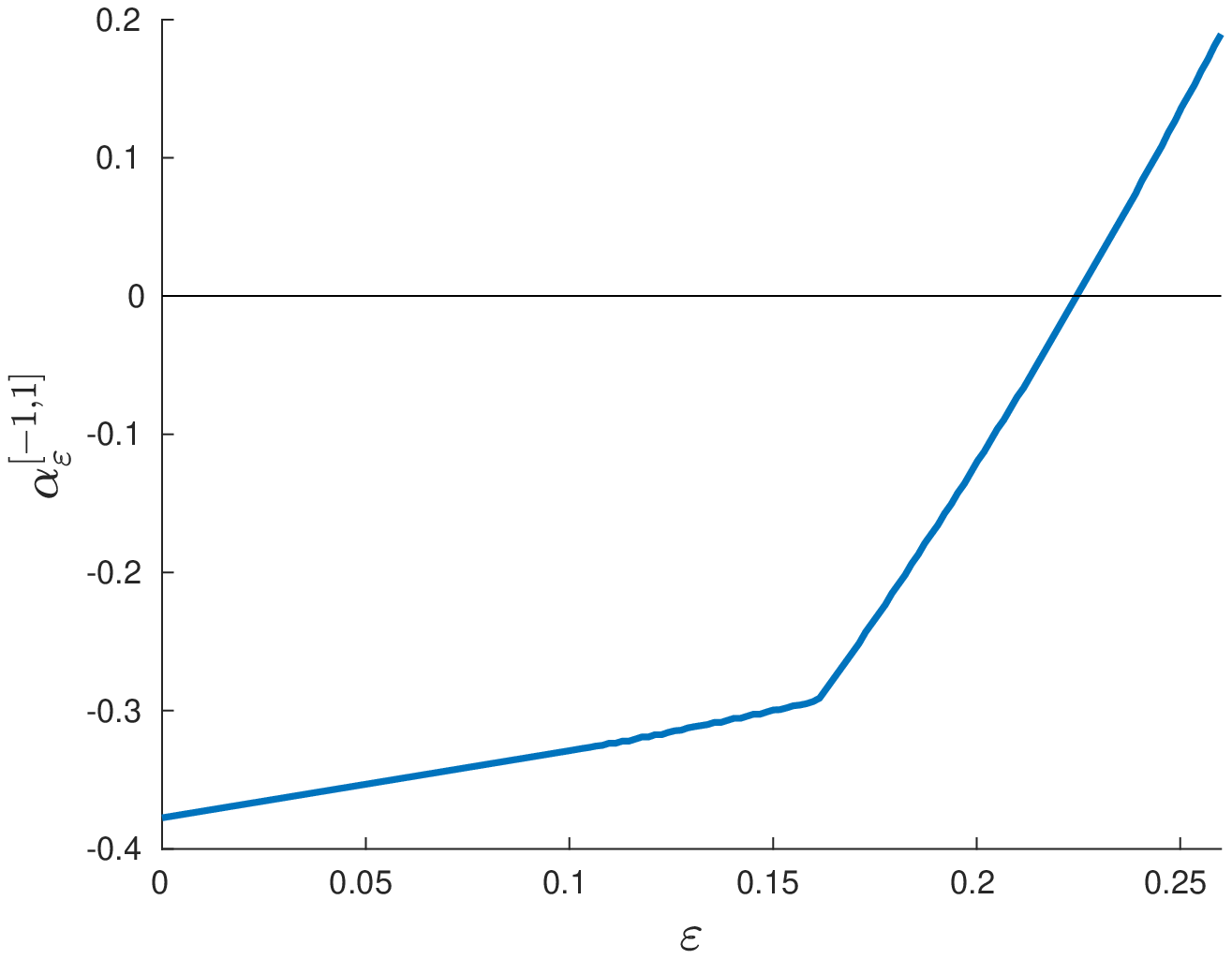}
		
		\end{minipage}\hfill
		\begin{minipage}[b]{0.49\linewidth}
			\centering
			\includegraphics[width=0.9\linewidth]{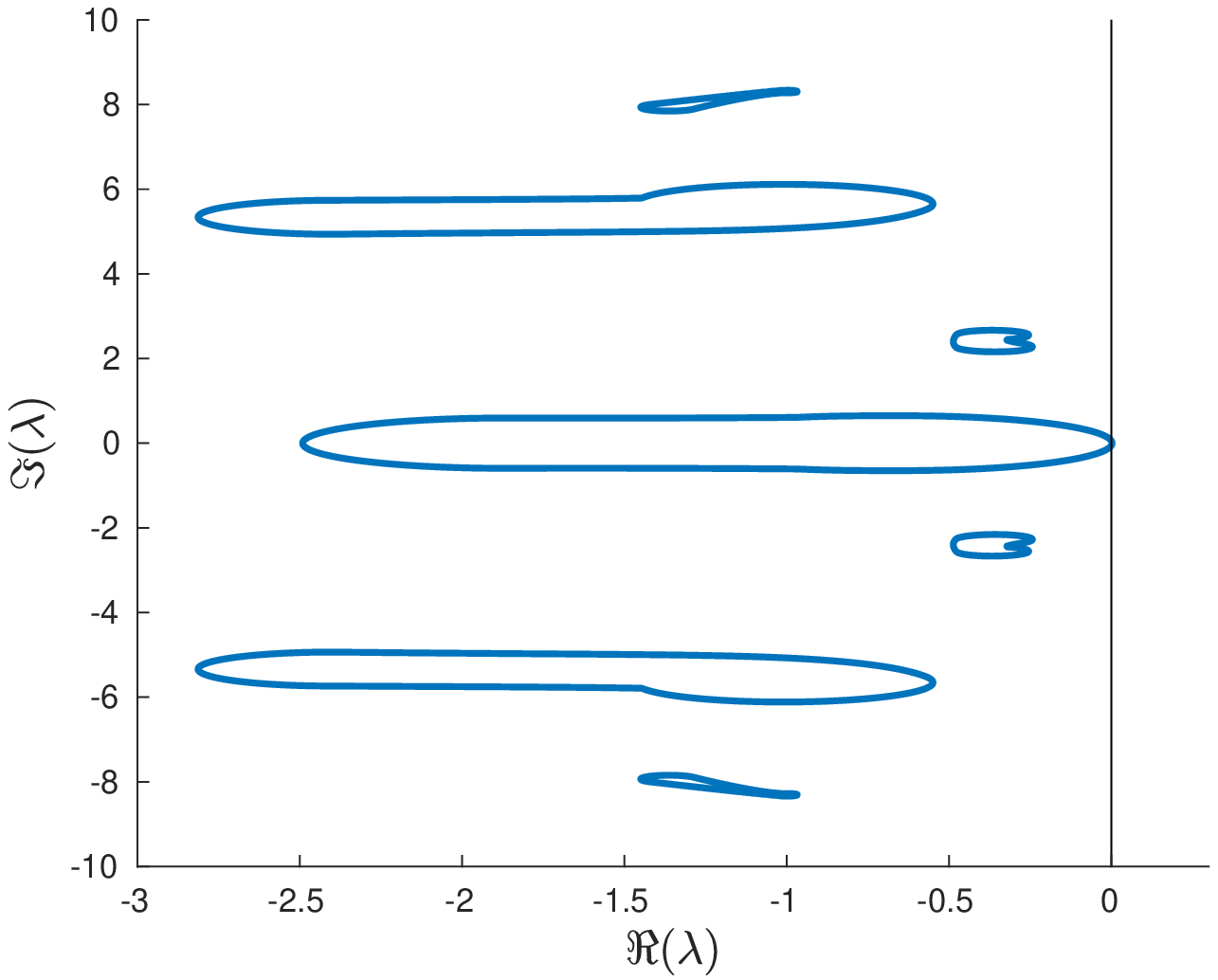}
			
		\end{minipage}\\[-7pt]
		\begin{minipage}[t]{0.45\linewidth}
				\caption{The pseudo-spectral abscissa of \eqref{pa_eq:example_characteristic_matrix} for $[a,\,b]$ equal to $[-1,\,1]$ in function of $\varepsilon$.}
			\label{pa_fig:pseudo_spectral_abscissa}
		\end{minipage}\hfill
			\begin{minipage}[t]{0.45\linewidth}
	\caption{The part of the spectral value set of \eqref{pa_eq:example_characteristic_matrix} in $[-3,0.3]\times\jmath[-10,10]$ for $[a,\,b]$ equal to $[-1,\,1]$ and $\varepsilon = 0.22491$.}
	\label{pa_fig:spectral_value_set2}
	\end{minipage}
	\end{figure}
\end{example}

We are now ready to state the relation between the robust $\mathcal{H}_{\infty}$-norm associated with uncertain system \eqref{pa_eq:uncertain_system} and the robust stability radius of \eqref{pa_eq:uncertain_DEP}.

\begin{theorem}
	\label{pa_th:rob_hinf_dist_ins}
	If uncertain system \eqref{pa_eq:uncertain_system} is internally exponentially stable for all $\lambda \in [a,\,b]$, its associated robust $\mathcal{H}_{\infty}$-norm is equal to the reciprocal of the robust stability radius of \eqref{pa_eq:uncertain_DEP}.
\end{theorem}
\begin{proof}
	The following proof is a simplification of the result in \cite{pa:Appeltans2019}.  	\\
	A complex number $\jmath\omega$ lies in $\Lambda_{\varepsilon}^{[a,\,b]}$ if and only if there exist $\lambda \in [a,\,b]$ and $\Delta\in\mathcal{B}_{\|\cdot\|_2\leq\varepsilon}^{\mathbb{C}^{m\times p}}$ such that
	\[	
	\det\big(M(\jmath \omega;\lambda,\Delta)\big) = \det\Big(I_n\jmath\omega - \textstyle\sum_{r=0}^{R} (H_r+ \lambda~ G_r ) e^{-\jmath\omega \tau_r}- B_w\Delta C_z\Big) = 0 \text{.}
	\]
	Because we required that \eqref{pa_eq:uncertain_system} is internally exponentially stable for all $\lambda \in [a,\,b]$, this is equivalent with
	\[
	\det\Big(I_n-\big(I_n\jmath\omega - \textstyle\sum_{r=0}^{R} (H_r+ \lambda G_r )  e^{-\jmath\omega \tau_r}\big)^{-1} B_w\Delta C_z \Big) = 0 \text{.}
	\]
	By the Weinstein-Aronszajn identity, this last equality can be rewritten as
	\[
	\begin{aligned}
	\det\Big(I\minus C_z\big(I\jmath\omega \minus \textstyle\sum_{r=0}^{R} (H_r+ \lambda G_r )  e^{\negative\jmath\omega \tau_r}\big)^{\negative 1} B_w\Delta  \Big) &= \det\big(I\minus T(\jmath\omega;\lambda)\Delta  \big) \\ &= 0.	
	\end{aligned}
	\]
	The characterization of the robust stability radius in \eqref{pa_eq:dist_ins2} can thus be rewritten as 
	\[
	r_{[a,\,b]} = \min_{\omega \in \mathbb{R}^{+}}~\min_{\lambda\in[a,\,b]}~ \min_{\Delta \in \mathbb{C}^{m \times p}}\Big\{\|\Delta\|_2: \det\big(I-T(\jmath\omega;\lambda)\Delta  \big) = 0 \Big\} \text{.}
	\]
	Using $\min_{\Delta \in \mathbb{C}^{m \times p}} \{\|\Delta\|_2: \det\left(I-M\Delta  \right) = 0 \} = \sigma_1\left(M\right)^{-1}$ from \cite{pa:Packard1993} one finds that
	\[
	\begin{aligned}
	r_{[a,\,b]} &= \min_{\substack{\lambda\in[a,\,b]\\\omega \in \mathbb{R}^{+}}} \Big( \sigma_1\big(T(\jmath\omega;\lambda)\big)\Big)^{-1} \\
	& = \Big(\max_{\substack{\lambda\in[a,\,b]\\\omega \in \mathbb{R}^{+}}} \sigma_1\big(T(\jmath\omega;\lambda)\big)\Big)^{-1} 
	= \left(\|T(\cdot;\cdot)\|_{\mathcal{H}_{\infty}}^{[a,\,b]}\right)^{-1},
	\end{aligned}
	\]
	which concludes the proof.
\end{proof}
\begin{remark}
	The presented relation can be generalized to systems with uncertainties on the delays, multiple uncertainties and other uncertainty structures, such as full block and diagonal uncertainties. Also systems with delays and uncertainties in the input, output and direct feed-through terms can be considered. For more information, see \cite{pa:Appeltans2019}.
\end{remark}

\subsection{Numerical algorithm}
\label{pa_subsec:numerical_algorithm}
 This subsection presents a numerical algorithm to compute the robust stability radius. Once this quantity is found, the robust $\mathcal{H}_{\infty}$-norm associated with \eqref{pa_eq:uncertain_system} follows immediately from \Cref{pa_th:rob_hinf_dist_ins}.

By \eqref{pa_eq:dist_ins}, the robust stability radius is the zero-crossing of the function $\mathbb{R}^{+}\ni\varepsilon\mapsto\alpha_{\varepsilon}^{[a,\,b]}$.  This zero-crossing can be found using the Newton-Bisection method, see \cite[Chapter~9.4]{pa:PressWilliamH1996NriF} for a reference implementation. This root finding method requires the evaluation of both $\alpha_{\varepsilon}^{[a,\,b]}$ and its derivative with respect to $\varepsilon$ for given $\varepsilon$ (whenever this derivative exists). The quantity $\alpha_{\varepsilon}^{[a,\,b]}$ can be computed using the method presented in \cite{pa:Appeltans2019}, which notes that $\alpha_{\varepsilon}^{[a,\,b]}$ is the solution of the following optimization problem:
 \begin{equation}
 \label{pa_eq:manifold_optimisation} 
 \begin{aligned}
 &\underset{s,\,\lambda,\,\Delta}{\text{max}}
 &&\Re\left(s\right),\\
 &\text{subject to}
 &&\det\big(M(s;\lambda,\Delta)\big) = 0,\\
 &&& \lambda \in [a,\,b], \\
 &&& \Delta\in \mathcal{B}_{\|\cdot\|_2\leq\varepsilon}^{\mathbb{C}^{m\times p}}.
 \end{aligned}
 \end{equation}
 Furthermore, the following proposition shows that there exists a $\Delta$ of rank one and norm $\varepsilon$ associated with local optima of this optimization problem. This result will allow us to reduce the search space for $\Delta$ to the space of matrices of rank one and spectral norm $\varepsilon$. We will denote this space as $\mathcal{B}_{\|\cdot\|_2=\varepsilon,\, \rank = 1}^{\mathbb{C}^{m\times p}}$.
 \begin{lemma}
 	Let $s^{\star} \not \in \Lambda^{[a,\,b]}_{0}$ be a local right-most point of $\Lambda_{\varepsilon}^{[a,\,b]}$, then there exist $\lambda^{\star} \in [a,\,b]$ and $\Delta^{\star}\! =\! \varepsilon u v^{H}$ with $u\!\in\!\mathbb{C}^{m}$, $v\!\in\!\mathbb{C}^{p}$ and $\|u\|_2\!=\!\|v\|_2\! =\! 1$ such that $\det\left(M(s^{\star};\lambda^{\star},\Delta^{\star})\right)\! =\! 0$.
 \end{lemma}
 \begin{proof}
 	Firstly, using similar ideas as in the proof of \Cref{pa_th:rob_hinf_dist_ins} one can show that \[\Lambda_{\varepsilon}^{[a,\,b]} = \Lambda^{[a,\,b]}_{0}\cup\Big\{s\in \mathbb{C} \setminus \Lambda^{[a,\,b]}_{0}: \max_{\lambda \in [a,\,b]} \sigma_{1}\left(T(s;\lambda)\right) \geq \varepsilon^{-1} \Big\}\text{.}\]
 	Because $s^{\star}\not \in \Lambda^{[a,\,b]}_{0}$ and $s^{\star}$ is a right-most point of of $\Lambda_{\varepsilon}^{[a,\,b]}$, $s^{\star}$ must lie on the boundary of $\big\{s\in \mathbb{C} \setminus \Lambda^{[a,\,b]}_{0}: \max_{\lambda \in [a,\,b]} \sigma_{1}\left(T(s;\lambda)\right) \geq \varepsilon^{-1} \big\}$. Hence, it holds that there exists a $\lambda^{\star} \in [a,\,b]$ such that $\sigma_{1}\left(T(s^{\star};\lambda^{\star})\right) = \varepsilon^{-1}$.\\
 	Secondly, it can easily be verified that $\det\left(I-T(s^{\star};\lambda^{\star})(\varepsilon\nu\upsilon^{H})\right) = 0$ for $\upsilon$ and $\nu$ the left and right normalized singular vectors associated with $\sigma_1\big(T(s^{\star};\lambda^{\star})\big)$. Following the derivation in the first part of the proof of \Cref{pa_th:rob_hinf_dist_ins} in the opposite direction one finds $\det\left(M\left(s^{\star};\lambda^{\star},\varepsilon\nu\upsilon^{H}\right)\right) = 0$.
 \end{proof}
 
 Constrained optimization problem \eqref{pa_eq:manifold_optimisation} is solved using a projected gradient flow method. The idea of this approach is to define a flow in the space of permissible variables along which the objective function monotonically increases and whose attractive stationary points are (local) optimizers of the optimization problem. These stationary points can be found by choosing initial parameters and discretizing the resulting path till convergence to a stationary point using Euler's forward method. The step size is chosen such that the objective function monotonically increases along the discretized path. For more details on the usage of these methods for the computation of extreme points of spectral value sets we refer to \cite{pa:Borgioli2019},\cite{pa:Guglielmi2011} and \cite{pa:Guglielmi2013}.

In our case we are thus looking for a path $[0,+\infty) \ni \theta \mapsto (\lambda(\theta),\Delta(\theta)) \in [a,\,b] \times \mathcal{B}_{\|\cdot\|_2=\varepsilon,\,\rank = 1}^{\mathbb{C}^{m\times p}}$ such that the function 
\[
\theta \mapsto \max\{\Re(s) :  \det\left(M\big(s;\lambda(\theta),\Delta(\theta)\big)\right)=0\}
\]
is monotonically increasing and such that the (local) optimizers of \eqref{pa_eq:manifold_optimisation} appear as attractive stationary points. Furthermore, to improve computational performance we employ an explicit decomposition of $\Delta(\theta)$ as \linebreak$\varepsilon u(\theta) v(\theta)^H$ inspired by \cite{pa:Guglielmi2011}. Here we consider the following flow where $\dot{u}$, $\dot{v}$ and $\dot{\lambda}$ denote the derivatives of $u$, $v$ and $\lambda$ with respect to $\theta$ and where the dependency of $u(\theta)$, $v(\theta)$ and $\lambda(\theta)$ on $\theta$ is omitted for notational convenience.

\begin{equation}
\label{pa_eq:resulting_path}
\left\{
\setlength{\arraycolsep}{2pt}
	\begin{array}{rl}
	\dot{u} =& \frac{\varepsilon}{\xi(\theta)}\Big(\!\left(I\minus uu^{H}\right) B_{w}^{T}\varphi(\theta)\psi(\theta)^{H}C_{z}^{T}v +  \frac{\jmath}{2}\Im\left(u^{H}B_{w}^{T}\varphi(\theta) \psi(\theta)^{H}C_{z}^{T}v\right)u\Big) \\ 
	\dot{v} =& \frac{\varepsilon}{\xi(\theta)}\Big(\!\left(I\minus vv^{H}\right)C_{z}^{\hphantom{T}}\psi(\theta)\varphi(\theta)^{H}B_{w}^{\hphantom{T}}u +\frac{\jmath}{2}\Im\left(v^{H}C_{z}^{\hphantom{T}}\psi(\theta) \varphi(\theta)^{H}B_{w}^{\hphantom{T}}u\right)v\hspace{0.01cm}\Big)  \\[10pt]
	\dot{\lambda} =& \left\{\begin{array}{l}
	\begin{array}{l}
	0  \hspace{0.6cm} \lambda = b \text{~and~} \left( \sum_{r=0}^{R} \Re\left(\varphi(\theta)^{H}G_r \psi(\theta) e^{-s(\theta)\tau_k}\right)\right)>0\\
	0  \hspace{0.6cm} \lambda = a \text{~and~} \left( \sum_{r=0}^{R} \Re\left(\varphi(\theta)^{H}G_r \psi(\theta) e^{-s(\theta)\tau_k}\right)\right)<0
	\end{array} \\
	\frac{1}{\xi(\theta)}\sum\limits_{r=0}^{R} \Re\left(\varphi(\theta)^{H}G_r \psi(\theta) e^{-s(\theta)\tau_k}\right) \hspace{2cm} \text{otherwise}
	\end{array} 
	\right.
	\end{array}
	\right.
\end{equation}
with $s(\theta)$ the right-most characteristic root of $M\big(s;\lambda(\theta),\varepsilon u(\theta) v(\theta)^{H}\big)$, and $\varphi(\theta)$ and $\psi(\theta)$ the associated left and right eigenvectors, normalized such that
 \[
\xi(\theta) = \varphi(\theta)^{H}\Big(I+\textstyle\sum_{r=0}^{R}\left(H_r + \lambda(\theta) G_{r}\right)\tau_r e^{-s(\theta) \tau_r}\Big)\psi(\theta)  > 0.
\] 

These paths are a combination of the paths presented in \cite{pa:Borgioli2019} (computing the pseudo-spectral abscissa for real-valued Frobenius norm bounded uncertainties) and \cite{pa:Guglielmi2011} (computing the pseudo-spectral abscissa for complex-valued spectral norm bounded uncertainties). The fulfillment of the constraints, the monotonous increase of the objective function and the (local) optimality of stationary points follows from these works. To conclude, we give some intuition behind \eqref{pa_eq:resulting_path}: the right-hand side can be interpreted as a projection of the gradient of the objective function on the tangent space of the feasible set. This projection step is needed to ensure that the variables remain within the feasible set.

The algorithm for numerically solving \eqref{pa_eq:manifold_optimisation} is summarized in \Cref{pa_alog:geometric_optimization}, where $s_{R}\big(M(s;\lambda,\varepsilon u v^{H})\big)$ gives the right-most characteristic root of \linebreak $M(s;\lambda,\varepsilon u v^{H})$ and $\dot{u}_k$, $\dot{v}_k$ and $\dot{\lambda}_k$ correspond to the right-hand side of \eqref{pa_eq:resulting_path} evaluated at $u_k$, $v_k$ and $\lambda_k$.
\begin{algorithm}
	\caption{Discretization algorithm to solve \eqref{pa_eq:manifold_optimisation}.}
	\label{pa_alog:geometric_optimization}
	\begin{algorithmic}
		\State $k\gets 0$ and choose initial $\lambda_0$, $u_0$, $v_0$
		\State $s_{-1}\gets -\infty$ and $s_0 \gets s_{R}\left(M\big(s;\lambda_0,\varepsilon u_0 v_0^{H}\big)\right)$
		\While {$|s_{k}-s_{k-1}| > \text{tol} \cdot \frac{|s_{k}+s_{k-1}|}{2}$}
		\State 
			Find $h$ such that $\Re\Big(s_{R}\left(M\big(s;\lambda_k\plus h\dot{\lambda}_k,\varepsilon (u_k \plus h\dot{u}_k) (v_k\plus h\dot{v}_k)^{H}\big)\right)\Big)\geq$ \par 
			\hskip\algorithmicindent $\Re\Big(s_{R}\big(M(s;\lambda_k,\varepsilon u_k v_k^{H})\big)\Big)$
		\normalsize
		\If {No $h>tol_h$ is found} stop.
		\Else
		\State $\lambda_{k+1} \gets \lambda_k+h\dot{\lambda}_k$; $\lambda_{k+1} \gets \max\{a,\min\{\lambda_{k+1},b\}\}$;
		\State $\rlap{$u_{k+1}$}\phantom{\lambda_{k+1}} \gets \phantom{\lambda_k}\llap{$u_k$} + h \dot{u}_k$; $\rlap{$u_{k+1}$}\phantom{\lambda_{k+1}} \gets \frac{u_{k+1}}{\|u_{k+1}\|_2}$; 
		\State $\rlap{$v_{k+1}$} \phantom{\lambda_{k+1}}  \gets \phantom{\lambda_k}\llap{$v_k$}+h\dot{v}_k$; $\rlap{$v_{k+1}$}\phantom{\lambda_{k+1}} \gets \frac{v_{k+1}}{\|v_{k+1}\|_2}$; 
		\State $\rlap{$s_{k+1}$}\phantom{\lambda_{k+1}} \gets s_{R}\left(M\big(s;\lambda_{k+1},\varepsilon u_{k+1} v_{k+1}^{H}\big)\right)$;
		\State $k\gets k+1$;
		\EndIf
		\EndWhile
	\end{algorithmic}
\end{algorithm}

\begin{remark}
	To compute the right-most characteristic root, we use the algorithm presented in \cite{pa:Wu2012}. This algorithm exploits the relation between a non-linear delay eigenvalue problem and the linear eigenvalue problem corresponding to the solution operator of the associated delay differential equation. More precisely, this method uses a spectral discretization of the solution operator to approximate the characteristic roots. These roots are subsequently refined by applying Newton corrections based on the original non-linear eigenvalue problem formulation. This methods is however restricted to small problems. For large sparse matrices one could use iterative methods such as \cite{pa:Jarlebring2010} and \cite{pa:Guttel2014} to compute the right-most characteristic roots.
\end{remark}

Once $\alpha_{\varepsilon}^{[a,\,b]}$ is computed, its derivative with respect to $\varepsilon$ can be computed almost everywhere as shown in the following proposition.
\begin{proposition}
	Let $\lambda^{\star}$ and $\Delta^{\star} = \varepsilon u^{\star} {v^{\star}}^{H}$ be the unique optimizers of \eqref{pa_eq:manifold_optimisation} and assume that the right-most characteristic root of $M(s;\lambda^{\star}, \Delta^{\star})$ is simple, then
	\[
	\dfrac{d\alpha_{\varepsilon}^{[a,\,b]}}{d\varepsilon} = \dfrac{\Re\left({\varphi^{\star}}^{H}B_{w} u^{\star} {v^{\star}}^{H} C_{z}\psi^{\star}\right)}{{\varphi^{\star}}^{H}(I+\sum_{r=0}^{R}\left(A_k + \lambda^{\star}G_{r}\right)\tau_{r}e^{-\tau_r s^{\star}})\psi^{\star}}\text{,}
	\] 
	with $s^{\star}$ the right-most characteristic root of $M(s;\lambda^{\star}, \Delta^{\star})$ and $\phi^{\star}$ and $\psi^{\star}$ its corresponding left and right eigenvectors, normalized such that the denominator is real and positive.
\end{proposition}
\begin{proof}
	Similar as in \cite[Theorem~2]{pa:Borgioli2019}.
\end{proof}

\section{A scalable $\hinf$-controller synthesis method}
\label{pa_sec:controller_design}
In this section we will describe a controller synthesis method for networked systems of form \eqref{pa:state_space_overall} whose associated adjacency matrix fulfills \Cref{pa_ass:adjacency_matrix_unitary}. The idea behind the presented method is to find a suitable controller by directly minimizing \eqref{pa_eq:rob_hinf_decoupled} in the controller parameters $\mathbf{p}$. Or in other words, we look for controller parameters $\mathbf{p}^{\star}$ that fulfill
 \begin{equation}
 \label{pa_eq:minimisation_problem}
\mathbf{p}^{\star} \in \underset{\mathbf{p}}{\arg\min}\|\hat{T}_{\hat{w}\hat{z}}(\cdot;\mathbf{p},\cdot)\|_{\mathcal{H}_{\infty}}^{[a,\,b]}.\\
\end{equation}
Note however that $\|\hat{T}_{\hat{w}\hat{z}}(\cdot;\mathbf{p},\cdot)\|_{\mathcal{H}_{\infty}}^{[a,\,b]}$ is only an upper bound for the actual \hinfnorm{} of \eqref{pa_eq:overall_tf}. The resulting control parameters will therefore in most cases not minimize the actual \hinfnorm{} of the system, but this methodology has the advantage that its computation cost does not depend on the number of subsystems. Furthermore, if \Cref{pa_ass:adjacency_matrix_unitary} remains valid with $a$ and $b$ independent of the number of subsystems, then the resulting controller guarantees a level of disturbance rejection even if the number of subsystems is unknown.

The minimization of \eqref{pa_eq:minimisation_problem} is however not trivial, as the robust \hinfnorm{} may be a non-smooth and non-convex function of the controller parameters even if the controller matrices are analytic functions of the controller parameters. This precludes the usage of standard optimization methods. Therefore, we will use HANSO \cite{pa:Overton2009}, which implements a combination of BFGS with weak Wolfe line search and gradient sampling. Furthermore, to decrease the chance of ending up at a local optimum we will restart the optimization algorithm from several initial points. The optimization procedure requires the evaluation of both the objective function and its derivative with respect to the control parameters whenever this derivative exists. To evaluate the objective function, we use the method presented in \Cref{pa_sec:computing_rob_hinf}. The derivative with respect to the control parameters follows from the following proposition.

\begin{proposition}
 	If there exists a unique pair \[
	(\omega^{\star},\lambda^{\star}) \in \underset{\lambda \in [a,\,b],~\omega \in \mathbb{R}^{+}}{\argmax}  \sigma_{1}\left(T(\jmath\omega;\mathbf{p},\lambda)\right)
	\] and if $\sigma_{1}\left(T(\jmath\omega^{\star};\mathbf{p},\lambda^{\star})\right)$ is simple,  then	the function $\mathbf{p}\mapsto\|\hat{T}_{\hat{w}\hat{z}}(\cdot;\mathbf{p},\cdot)\|_{\mathcal{H}_{\infty}}^{[a,\,b]}$ is differentiable at $\mathbf{p}$, with
	\[
	\dfrac{d}{d\mathbf{p}}\|T(\cdot;\mathbf{p},\cdot)\|_{\mathcal{H}_{\infty}}^{[a,\,b]} = \Re\left(u^{H}\left(\frac{\partial }{\partial \mathbf{p}}T(\jmath \omega^{\star};\mathbf{p},\lambda^{\star}) \right)v\right)\text{,}
	\]
	in which $u$ and $v$ are the normalized left and right singular vectors associated with $\sigma_{1}\big(T(\jmath\omega^{\star};\mathbf{p},\lambda^{\star})\big)$, respectively.
\end{proposition}
Finally, to start the optimization process we need initial controller parameters $\mathbf{p}$ for which the system is internally exponentially stable for all $\lambda \in [a,\,b]$. To find such a starting point, we use the method presented in \cite{pa:Dileep2018}.

\section{Example}
\label{pa_sec:example}
	In this example we consider a networked system that consists of $N$ frictionless carts that are interconnected using identical springs and that each balance an inverted pendulum. Furthermore, the first and the last cart are connected to the wall using additional (but identical) springs. This set-up is illustrated in \Cref{pa_fig:setup} for $N$ equal to 3. The dynamics of an individual cart-pendulum subsystem are governed by the following non-linear delay differential equations
	\[
	\left\{
	\begin{array}{l}
	(M+m)\ddot{x}_{j}(t)  +~ml\cos\big(\theta_j(t)\big)\ddot{\theta}_j(t)-ml\sin\big(\theta_j(t)\big) \big(\dot{\theta_j}(t)\big)^{2} ~+ \\ [2pt]  k\big(x_{j}(t)-x_{j+1}(t)\big)+k\big(x_{j}(t)-x_{j-1}(t)\big) - u_{j}(t-\tau_{u})-w_{j,1}(t) = 0
	\\ [4pt]
	l\ddot{\theta}_j(t)-g\sin\big(\theta_j(t)\big)  + \ddot{x}_j(t) \cos\big(\theta_j(t)\big)-w_{j,2}(t) = 0  
	\end{array}
	\right.
	\]
	for $j=1,\dotsc,N$ and $M$ the mass of the individual carts, $m$ the mass of the pendulum's bob which is connected to the cart using a massless rod of length $l$, $k$ the spring constant, $u_j$ a controllable force that acts on the cart with an input delay $\tau_u$, $w_{j,1}$ and $w_{j,2}$ external disturbances, $\theta_{j}$ the angular displacement of the pendulum of cart $j$, $x_{j}$ the horizontal displacement of the $j$\textsuperscript{th} cart's center with respect to its equilibrium position and $x_{0}=x_{N+1} = 0$.
	By a linearization around the equilibrium point $(x_{j},\dot{x}_j,\theta_j,\dot{\theta}_j) = (0,0,0,0)$ and choosing $\begin{bmatrix}
		x_{j}(t) & \theta_{j}(t)
	\end{bmatrix}^{T}$ as both measured and performance output, we obtain the following linear state-space model of form \eqref{pa_eq:subsysem}:
	\begin{equation}
	\label{pa_eq:subsystem}
	\setlength{\arraycolsep}{2pt}
	\begin{array}{rcl}
	\dot{x}_{j}(t)&
	=&
	\begin{bmatrix}
	0 & 1 & 0 & 0 \\
	\negative\frac{2k}{M}& 0 & \negative\frac{mg}{M} & 0\\
	0 & 0 & 0 & 1 \\
	\frac{2k}{Ml} & 0 & \frac{g}{l}\plus\frac{mg}{Ml} & 0
	\end{bmatrix}\! x_{j}(t)
	+ 
	\begin{bmatrix}
	0 \\
	\frac{1}{M} \\
	0 \\
	\frac{\negative1}{Ml}
	\end{bmatrix}\!
	u_{j}(t \minus \tau_{u})
	+
	\begin{bmatrix}
	0 \\
	\frac{k}{M} \\
	0 \\
	\frac{\negative k}{Ml}
	\end{bmatrix}\!
	u^{n}_{j}(t)
	+ \\[25pt]
	&&
	\begin{bmatrix}
	0 & 0 \\
	\frac{1}{M}  & \frac{\negative m}{M}\\
	0 & 0 \\
	\frac{\negative 1}{Ml}  & \frac{1}{l}\plus\frac{m}{Ml}
	\end{bmatrix}\!
	w_{j}(t)
	\end{array}
	\end{equation}
	\[
	\setlength{\arraycolsep}{2pt}
	y_{j}(t) = \begin{bmatrix}
	1 & 0 & 0 & 0 \\
	0 & 0 & 1 & 0
	\end{bmatrix} x_{j}(t), ~~ y^{n}_{j}(t) = \begin{bmatrix}
	2 & 0 & 0 & 0
	\end{bmatrix}x_{j}(t),~~
	z_{j}(t) = \begin{bmatrix}
	1 & 0 & 0 & 0 \\
	0 & 0 & 1 & 0
	\end{bmatrix} x_{j}(t)
	\]
	with,
	\[
	u^{n}_{j}(t) = \sum_{i=1}^{N} P_{j,i}^{N}~y_{i}^{n}(t)\text{,}
	\]
	and
	\[
	P_N = 
	P_N^{\,\text{line}} = \begin{bmatrix}
	0 & 0.5 &  &    & \\
	0.5 & 0 & 0.5 &   & \\
	 & \ddots & \ddots & \ddots &  \\
	 &  &  0.5 & 0 & 0.5 \\
	 &  &  & 0.5 & 0
	\end{bmatrix}.
	\]
	The input signal $u_j$ is generated using a controller of form \eqref{pa_eq:controller} with $n_c$ equal to 2. Furthermore, we will allow communication between the controllers of neighboring subsystems and as control parameters $\mathbf{p}$ we choose the elements of the control matrices.

		\hfill\allowbreak
	\begin{figure}[!htbp]
		\centering
		\begin{tikzpicture}[scale=0.8]
		\draw (0,3) -- (0,0) -- (12,0) -- (12,3);
		\fill[LineSpace=9pt, pattern=my north west lines,pattern color=gray] (0,3) -- (0,0) -- (12,0) -- (12,3) -- (12.5,3) -- (12.5,-0.5) -- (-0.5,-0.5) -- (-0.5,3) -- cycle;
		
		\draw[decoration={aspect=0.4, segment length=2mm, amplitude=2mm,coil},decorate] (0.15,0.9) -- (1.5,0.9); 
		\draw (0,0.9) -- (0.15,0.9);
		\node (k) at (0.75,0.45) {$k$};

		\draw (2,0.3) circle (0.3);
		\draw (3,0.3) circle (0.3);
		\draw (1.5,0.6) rectangle (3.5,1.2);
		\node (M) at (2.5,0.9) {$M$};
		\filldraw (3.25,3) circle (0.15);
		\draw (2.5,1.2) -- (3.25,3);
		\node (m) at (3.35,3.35) {$m$};
		\node [rotate=67.38] (l) at (3.15,2.1) {$l$};
		\draw[dashed] (2.5,1.2) -- (2.5,2.5);
		\draw [domain=67.38:90] plot ({2.5+0.7*cos(\x)}, {1.2+0.7*sin(\x)});
		\node (theta) at (2.72,2.3) {$\theta_1$};
		\draw (2.5,-0.7) -- (2.5,-1.1);
		\draw [-{>[scale=2]}] (2.5,-0.9) -- (3.5,-0.9);
		\node (x) at (3,-1.3) {$x_1$};
		\draw (3.5,1.2) -- (3.5,1.6);
		\draw[->] (3.5,1.4) -- (4.3,1.4);
		\node (F) at (3.9,1.7) {$u_1$};
		
		\draw[decoration={aspect=0.4, segment length=2mm, amplitude=2mm,coil},decorate] (3.65,0.9) -- (5,0.9); 
		\draw (3.5,0.9) -- (3.65,0.9);
		\node (k) at (4.25,0.45) {$k$};
		
		\draw (5.5,0.3) circle (0.3);
		\draw (6.5,0.3) circle (0.3);
		\draw (5,0.6) rectangle (7,1.2);
		\node (M) at (6,0.9) {$M$};
		\filldraw (6.75,3) circle (0.15);
		\draw (6,1.2) -- (6.75,3);
		\node (m) at (6.85,3.35) {$m$};
		\node [rotate=67.38] (l) at (6.65,2.1) {$l$};
		\draw[dashed] (6,1.2) -- (6,2.5);
		\draw [domain=67.38:90] plot ({6+0.7*cos(\x)}, {1.2+0.7*sin(\x)});
		\node (theta) at (6.22,2.3) {$\theta_2$};
		\draw (6,-0.7) -- (6,-1.1);
		\draw [-{>[scale=2]}] (6,-0.9) -- (7,-0.9);
		\node (x) at (6.5,-1.3) {$x_2$};
		\draw (7,1.2) -- (7,1.6);
		\draw[->] (7,1.4) -- (7.8,1.4);
		\node (F) at (7.4,1.7) {$u_2$};
		
		\draw[decoration={aspect=0.4, segment length=2mm, amplitude=2mm,coil},decorate] (7.15,0.9) -- (8.5,0.9); 
		\draw (7,0.9) -- (7.15,0.9);
		\node (k) at (7.75,0.45) {$k$};
		
		\draw (9,0.3) circle (0.3);
		\draw (10,0.3) circle (0.3);
		\draw (8.5,0.6) rectangle (10.5,1.2);
		\node (M) at (9.5,0.9) {$M$};
		\filldraw (10.25,3) circle (0.15);
		\draw (9.5,1.2) -- (10.25,3);
		\node (m) at (10.35,3.35) {$m$};
		\node [rotate=67.38] (l) at (10.15,2.1) {$l$};
		\draw[dashed] (9.5,1.2) -- (9.5,2.5);
		\draw [domain=67.38:90] plot ({9.5+0.7*cos(\x)}, {1.2+0.7*sin(\x)});
		\node (theta) at (9.72,2.3) {$\theta_3$};
		\draw (9.5,-0.7) -- (9.5,-1.1);
		\draw [-{>[scale=2]}] (9.5,-0.9) -- (10.5,-0.9);
		\node (x) at (10,-1.3) {$x_3$};
		\draw (10.5,1.2) -- (10.5,1.6);
		\draw[->] (10.5,1.4) -- (11.3,1.4);
		\node (F) at (10.9,1.7) {$u_3$};
		
		\draw[decoration={aspect=0.4, segment length=2mm, amplitude=2mm,coil},decorate] (10.65,0.9) -- (12,0.9); 
		\draw (10.5,0.9) -- (10.65,0.9);
		\node (k) at (11.25,0.45) {$k$};
		\end{tikzpicture}	
		\caption{Schematic representation of the considered set-up for $N= 3$.}
		\label{pa_fig:setup}
	\end{figure}
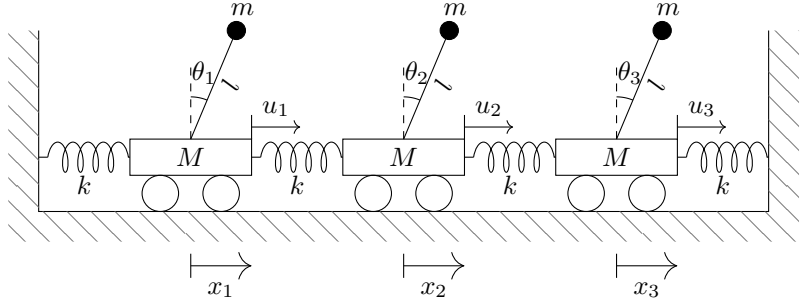
	\vspace{1em}
	
	As mentioned in \Cref{pa_ex:adjacency}, the eigenvalues of $P_N^{\,\text{line}}$ are restricted to the interval $[-1,\,1]$ for all $N>1$. It thus follows from \Cref{pa_cor:upperbound_hinfnorm} that the robust \hinfnorm{} of a single subsystem forms an upper bound for the \hinfnorm{} of the overall network that holds independent of $N$. Furthermore, this upper bound can be computed efficiently using the method presented in Section~\ref{pa_sec:computing_rob_hinf}. 
	
	For the following parameter values: $M = 1$ \si{kg}, $m = 0.05$ \si{kg}, $k=1$ \si{N/m}, $l=1$ \si{m}, $g=9.8$ \si{m/s^2}, $\tau_{u} = 0.1$ \si{s} and $\tau_{u_{nc}} = 0.2$ \si{s},  we apply the procedure of \Cref{pa_sec:controller_design} to find the controller parameters that minimize this upper bound. The obtained controller matrices (rounded to four digits accuracy) are:
	\begin{equation}
	\label{pa_eq:resulting_controller}
	\setlength{\arraycolsep}{2pt}
	\renewcommand{\arraystretch}{1.2}
	\begin{array}{lclclcl}
	J_{\mathbf{p}^{\star}} &=& \begin{bmatrix}
	\negative 67.92 & \negative 313.2\\
	\negative 52.15 & \negative 306.8
	\end{bmatrix} & & F_{\mathbf{p}^{\star}} &=& \begin{bmatrix}
	\negative 407.4 & 1139 \\
	\phantom{\negative}141.8 & 1519
	\end{bmatrix} \\ [10pt]
	 F_{\mathbf{p}^{\star}}^{n} &=& \begin{bmatrix}
	\negative 33.74 & \negative 112.6 \\
	\negative 91.45 & \negative 147.9
	\end{bmatrix} & &
	L_{\mathbf{p}^{\star}} &=& \begin{bmatrix}
	\negative 143.3 & \negative 834.4
	\end{bmatrix} \\ [10pt]
	K_{\mathbf{p}^{\star}} & = & \begin{bmatrix}
	\phantom{\negative} 349.9 & 4172
	\end{bmatrix} & &  K_{\mathbf{p}^{\star}}^{n} & =& \begin{bmatrix}
	\negative 246.5 & \negative 406.9
	\end{bmatrix},
	\end{array}
	\end{equation}
	in which the subscript $\mathbf{p}^{\star}$ is used to indicate that these matrices correspond to the minimizing control parameters.
	 The corresponding upper bound for the \hinfnorm{} equals 0.512249 (rounded to 6 digits). \Cref{table:hinfnorm_N} gives the actual \hinfnorm{} of the closed loop system  with controller matrices \eqref{pa_eq:resulting_controller} for several $N$. We observe that these values lie close to the computed upper bound. This table also gives the time required by the algorithm described in \cite{pa:Gumussoy2011} to compute these values. As expected, the computation time grows roughly cubically with respect to $N$. For comparison, the computation time of the robust \hinfnorm{} of the associated uncertain subsystem equals 75.70 \si{s}. For system with large number of subsystems it is thus beneficiary to minimize this upper bound instead of the actual \hinfnorm{}.
	
	\begin{table}[!htbp]
		\centering
		\caption{The \hinfnorm{} (rounded to 6 digits accuracy) of the closed loop networked system for several $N$ and the computation time required by the algorithm described in \cite{pa:Gumussoy2011} to compute these values.}
		\label{table:hinfnorm_N}
		\renewcommand{\arraystretch}{1.3}
		\begin{tabular}{c|c|c|c|c}
			$N$ & 3 & 5 & 10 & 15  \\ \hline
			$\|T(\cdot;\mathbf{p}^{\star},N)\|_{\mathcal{H}_{\infty}}$& 0.512228 & 0.512239 & 0.512246 & 0.512247  \\\hline
			Computation time (\si{s}) & 17.39 & 91.56 & 806.3 & 2961
		\end{tabular}
	\end{table}
	
	Next, we will examine the disturbance rejection of the closed loop system for $N$ equal to 20. We simulate the system for $t\in[0,10]$ where both the state and the controller state are equal to $0$ for $t<0$, and each performance input is low pass filtered ($\omega_{cutoff}=6\pi$) Gaussian white noise which is scaled after filtering to have a root mean square energy of 0.1. \Cref{pa_fig:simulations} shows the performance inputs and outputs of the 10\textsuperscript{th} subsystem. We observe that the noise is well attenuated. The root mean square energy of $z_{10,1}$ and $z_{10,2}$ are equal to 0.03610 and  0.02094, respectively. 
	
	\begin{figure}[!htbp]
		\centering
		\begin{subfigure}{0.48\linewidth}
			\centering
			\includegraphics[width=\linewidth]{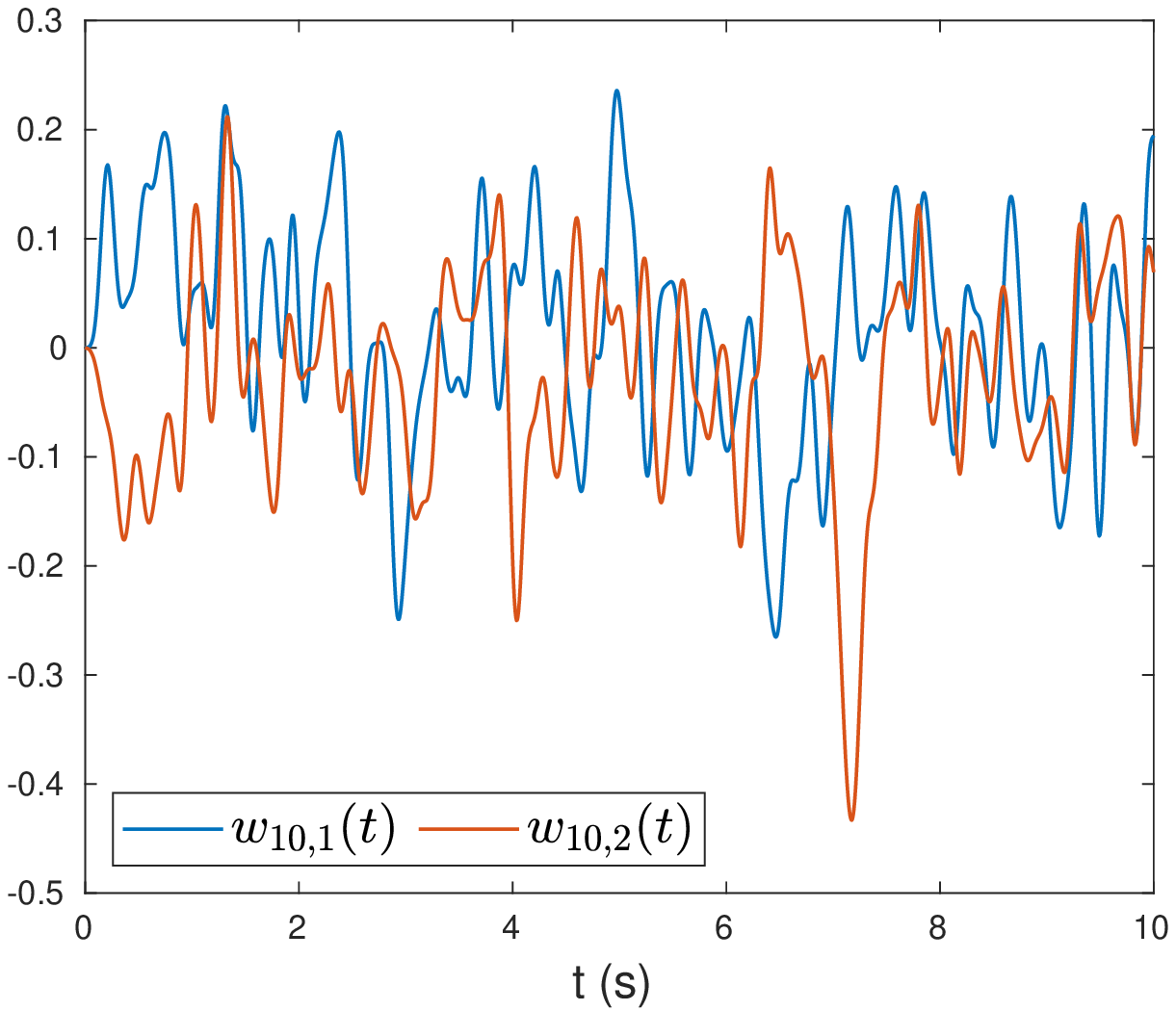}
		\end{subfigure}
		\begin{subfigure}{0.48\linewidth}
			\centering
			\includegraphics[width=\linewidth]{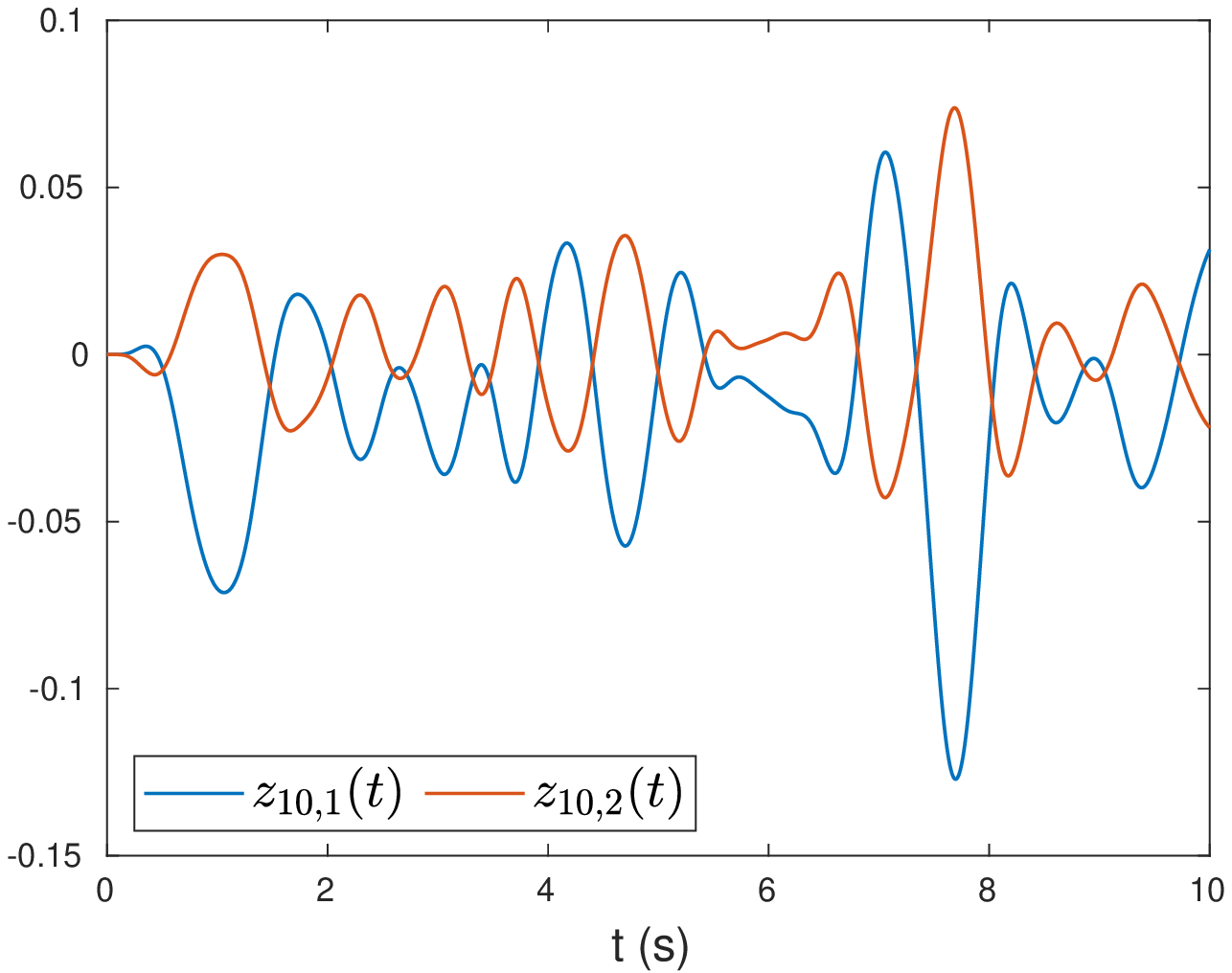}
		\end{subfigure}
		\caption{Simulation of the performance inputs (left) and outputs (right) of the 10\textsuperscript{th} subsystem for $t\in[0,10]$  of the closed loop system \eqref{pa_eq:subsystem}-\eqref{pa_eq:resulting_controller} for $N=20$. The state and the controller state are equal to zero for $t<0$. Each performance input is low pass filtered ($\omega_{cutoff}=6\pi$) Gaussian white noise, scaled after filtering to have a root mean square energy of 0.1.}
		\label{pa_fig:simulations}
	\end{figure}

	As explained in \Cref{pa_cor:upperbound_hinfnorm}, the robust \hinfnorm{} of an associated uncertain subsystem of form \eqref{pa:decoupled_system} was used as upper bound for the actual \hinfnorm{} of the networked system. \Cref{pa_fig:wcfreqresp} shows the ``worst-case gain function'' of this associated uncertain subsystem:
	 \begin{equation}
	 \label{pa_eq:wc_freq_responce}
	 \omega \mapsto \max_{\lambda \in [-1,\,1]} \sigma_1\left(\hat{T}_{\hat{w}\hat{z}}(\jmath\omega;\mathbf{p}^{\star},\lambda)\right)
	 \end{equation}
	 for $\omega \in [10^{-1},10^{2}]$. The robust \hinfnorm{}, which is equal to the maximal value of this worst-case gain function, is indicated with a dashed line. We observe that the worst-case gain function is flat and almost equal to the robust \hinfnorm{} for $\omega \in [0,10]$. This phenomenon is typical for the direct optimization framework.
	 
	\begin{figure}[!htbp]
		\centering
		\includegraphics[width=0.7\linewidth]{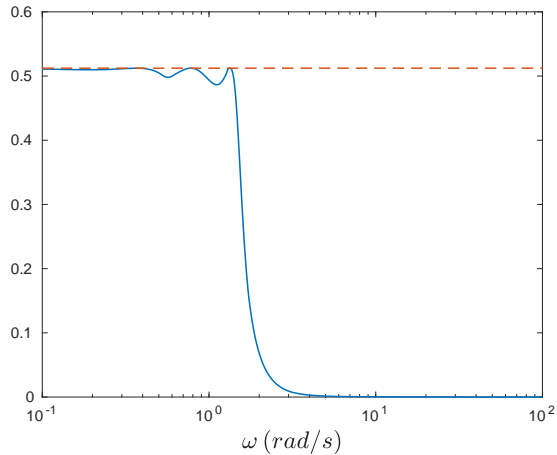}
		\caption{The worst-case gain function of the associated uncertain subsystem as defined in \eqref{pa_eq:wc_freq_responce} for $\omega \in [10^{-1},10^{2}]$. The robust \hinfnorm{}, $\|\hat{T}_{\hat{w}\hat{z}}(\cdot;\mathbf{p}^{\star},\cdot)\|_{\mathcal{H}_{\infty}}^{[-1,\,1]}$, is indicated with a dashed line.}
		\label{pa_fig:wcfreqresp}
	\end{figure}

\section{Conclusion}
\label{pa_sec:conclusion}
In this manuscript we introduced a scalable controller synthesis method for networked systems with identical subsystems and an interconnection topology that fulfills \Cref{pa_ass:adjacency_matrix_unitary}. Using the decoupling transformation of \Cref{sec:networked_systems} we arrived at an upper bound for the \hinfnorm{} which can be computed at a cost that does not dependent on the number of subsystems. This upper bound was the robust \hinfnorm{} of a single subsystem with an additional scalar uncertainty. Subsequently an algorithm to compute this robust \hinfnorm{} was discussed in Section~\ref{pa_sec:computing_rob_hinf}. Finally, Section~\ref{pa_sec:controller_design} showed how a controller that minimizes this upper bound can be synthesized using the direct optimization framework.

To conclude, we note that the presented method can be extended to more general identical subsystems, e.g. subsystems with direct feed-through terms. For such systems, however, the \hinfnorm{} is sensitive to infinitesimal delay changes and the strong \hinfnorm{} is a more appropriate performance measure \cite{pa:Gumussoy2011}. Furthermore, as in \Cref{sec:networked_systems}, one can show that the strong \hinfnorm{} of the overall system is upper bounded by the robust strong \hinfnorm{} of a single uncertain subsystem. This robust strong \hinfnorm{} can be computed using the method presented in \cite{pa:Appeltans2019}. 

\section*{Acknowledgment}
This work was supported by the project C14/17/072 of the KU Leuven Research Council and by the project G0A5317N of the Research Foundation-Flanders (FWO - Vlaanderen).


\begin{thebibliography}{10}
	
	\bibitem{pa:Appeltans2019}
	Pieter Appeltans and Wim Michiels.
	\newblock {A pseudo-spectra based characterisation of the robust strong
		H-infinity norm of time-delay systems with real-valued and structured
		uncertainties}.
	\newblock {\em Preprint online available on arXiv,}
	\newblock arXiv:1909.07778 [math.NA],
	\newblock 2019.
	
	\bibitem{pa:Borgioli2019}
	Francesco Borgioli and Wim Michiels.
	\newblock {A Novel Method to Compute the Structured Distance to Instability for
		Combined Uncertainties on Delays and System Matrices}.
	\newblock {\em IEEE Transactions on Automatic Control}, 65(4):1747--1754,
	2020.
	
	\bibitem{pa:Dileep2018}
	Deesh Dileep, Francesco Borgioli, Laurentiu Hetel, Jean~Pierre Richard, and Wim
	Michiels.
	\newblock {A scalable design method for stabilising decentralised controllers
		for networks of delay-coupled systems}.
	\newblock {\em IFAC-PapersOnLine}, 51(33):68--73, 2018.
	
	\bibitem{pa:Dileep2018a}
	Deesh Dileep, Ruben {Van Parys}, Goele Pipeleers, Laurentiu Hetel, Jean~Pierre
	Richard, and Wim Michiels.
	\newblock {Design of robust decentralised controllers for MIMO plants with
		delays through network structure exploitation}.
	\newblock {\em International Journal of Control}, 2018.
	
	\bibitem{pa:Guglielmi2011}
	Nicola Guglielmi and Christian Lubich.
	\newblock {Differential equations for roaming pseudospectra: paths to extremal
		points and boundary tracking}.
	\newblock {\em SIAM Journal on Numerical Analysis}, 49(3):1194--1209,  2011.
	
	\bibitem{pa:Guglielmi2013}
	Nicola Guglielmi and Christian Lubich.
	\newblock {Low-rank dynamics for computing extremal points of real
		pseudospectra}.
	\newblock {\em SIAM Journal on Matrix Analysis and Applications}, 34(1):40--66,
	2013.
	
	\bibitem{pa:Gumussoy2011}
	Suat Gumussoy and Wim Michiels.
	\newblock {Fixed-order H-Infinity control for interconnected systems using
		delay differential algebraic equations}.
	\newblock {\em SIAM Journal on Control and Optimization}, 49(5):2212--2238,
	2011.
	
	\bibitem{pa:Guttel2014}
	Stefan G{\"{u}}ttel, Roel {Van Beeumen}, Karl Meerbergen, and Wim Michiels.
	\newblock {NLEIGS: A Class of Fully Rational Krylov Methods for Nonlinear
		Eigenvalue Problems}.
	\newblock {\em SIAM Journal on Scientific Computing}, 36(6):A2842--A2864,
	2014.
	
	\bibitem{pa:Hilhorst2015}
	Gijs Hilhorst, Goele Pipeleers, Wim Michiels, and Jan Swevers.
	\newblock {Sufficient LMI conditions for reduced-order multi-objective H 2 / H
		$\infty$ control of LTI systems}.
	\newblock {\em European Journal of Control}, 23:17--25, 2015.
	
	\bibitem{pa:Jarlebring2010}
	Elias Jarlebring, Karl Meerbergen, and Wim Michiels.
	\newblock {A Krylov Method for the Delay Eigenvalue Problem}.
	\newblock {\em SIAM Journal on Scientific Computing}, 32(6):3278--3300, 
	2010.
	
	\bibitem{pa:Massioni2009}
	Paolo Massioni and Michel Verhaegen.
	\newblock {Distributed control for identical dynamically coupled systems: A
		decomposition approach}.
	\newblock {\em IEEE Transactions on Automatic Control}, 54(1):124--135, 2009.
	
	\bibitem{pa:Michiels2011}
	Wim Michiels.
	\newblock {Spectrum-based stability analysis and stabilisation of systems
		described by delay differential algebraic equations}.
	\newblock {\em IET Control Theory \& Applications}, 5(16):1829--1842, 2011.
	
	\bibitem{pa:Overton2009}
	Michael~L. Overton.
	\newblock {HANSO: a hybrid algorithm for nonsmooth optimization}.
	\newblock {\em Available online: }{\tt https://cs.nyu.edu/overton/software/hanso/},
	\newblock 2009.
	
	\bibitem{pa:Ozer2015}
	S.~Mert {\"{O}}zer and Altuʇ Iftar.
	\newblock {Decentralized controller design for time-delay systems by
		optimization}.
	\newblock {\em IFAC-PapersOnLine}, 28(12):462--467, 2015.
	
	\bibitem{pa:Packard1993}
	Andrew Packard and John~C. Doyle.
	\newblock {The complex structured singular value}.
	\newblock {\em Automatica}, 29(1):71--109, 1993.
	
	\bibitem{pa:PressWilliamH1996NriF}
	William~H. Press, Saul~A. Teukolsky, and William~T. Vetterling.
	\newblock {\em {Numerical recipes in Fortran 77 : the art of scientific
			computing}}.
	\newblock  Cambridge University press, Cambridge,
	2nd edition, 1996.
	
	\bibitem{pa:Wu2012}
	Zhen Wu and Wim Michiels.
	\newblock {Reliably computing all characteristic roots of delay differential
		equations in a given right half plane using a spectral method}.
	\newblock {\em Journal of Computational and Applied Mathematics},
	236(9):2499--2514, 2012.
	
	\bibitem{pa:Zhou1998}
	Kemin Zhou and John~C. Doyle.
	\newblock {\em {Essentials of robust control}}.
	\newblock Prentice Hall, Upper Saddle River, NJ, 1998.
	
\end{thebibliography}
\end{document}